\documentclass[12pt,draftcls,onecolumn]{IEEEtran}

\usepackage[utf8]{inputenc}
\usepackage{amsmath,amssymb,algorithm,algorithmic,graphicx}
\usepackage{color}
\usepackage{tabularx,eclbkbox,colortbl}

\newtheorem{theorem}{Theorem}

\newtheorem{proposition}{Proposition}

\newtheorem{lemma}{Lemma}
\newtheorem{corollary}{Corollary}
\newtheorem{assumption}{Assumption}
\newtheorem{definition}{Definition}
\newtheorem{example}{Example}
\newtheorem{scenario}{Scenario}

\renewcommand{\d}{\mathrm{d}}
\newcommand{\grad}{\mathop{\mathrm{grad}}}
\renewcommand{\div}{\mathop{\mathrm{div}}}
\newcommand{\per}{\mathop{\mathrm{Per}}}
\newcommand{\proj}{\mathop{\mathrm{proj}}\nolimits}
\newcommand{\prox}{\mathop{\mathrm{prox}}\nolimits}
\newcommand{\median}{\mathop{\mathrm{median}}}

\newcommand{\sign}{\mathop{\mathrm{sign}}}
\newcommand{\soft}{\mathop{\mathrm{soft}}\nolimits}

\newcommand{\tv}{\mathrm{TV}}
\newcommand{\one}{1}
\newcommand{\zero}{0}

\newcommand{\modif}[1]{{\color{black} #1}}

\title{
  Robust Consensus in Distributed Networks using Total Variation
}
\author{Walid Ben-Ameur, Pascal Bianchi and J\'er\'emie Jakubowicz%
}

\newcommand{\bR}{{\mathbb{R}}}
\newcommand{\bN}{{\mathbb{N}}}

\newcommand{\cU}{{\mathcal{U}}}

\newcommand{\cL}{{\mathcal{L}}}
\newcommand{\cC}{{\mathcal{C}}}
\newcommand{\bs}{\boldsymbol}

\begin{document}
\maketitle

\begin{abstract}
  Consider a connected network of agents endowed with local cost functions representing private objectives. Agents seek to find an agreement on some minimizer of the aggregate cost, by means of repeated communications between neighbors. Consensus on the average over the network, usually addressed by gossip algorithms, is a special instance of this problem, corresponding to quadratic private objectives. Consensus on the median, or more generally quantiles, is also a special instance, as many more consensus problems. In this paper we show that optimizing the aggregate cost function regularized by a total variation term has appealing properties. First, it can be done very naturally in a distributed way, yielding algorithms that are efficient on numerical simulations. Secondly, the optimum for the regularized cost is shown to be also the optimum for the initial aggregate cost function under assumptions that are simple to state and easily verifiable. Finally, these algorithms are robust to unreliable agents that keep injecting some false value in the network. This is remarkable enough, and is not the case, for instance, of gossip algorithms, that are entirely ruled by unreliable agents as detailed in the paper.
\end{abstract}

\section{Introduction}

Total variation has been widely used in the framework of inverse problems, where the aim is to recover a mathematical object that shows good regularity properties. An important landmark is~\cite{rudin1992nonlinear} that successfully applied total variation regularization to image denoising. In the previous decade the role of the $L_1$ norm was clearly connected to sparsity~\cite{candes2005decoding,candes2006robust,donoho2006compressed}. In this light, total variation promotes sparsity of the gradient; yielding locally flat objects. Curiously enough, total variation has been mainly developed in the continuous setting, where ``bounded variation'' functions and their properties are well studied~\cite{ambrosio2000functions}; while in the discrete setting, its properties on graphs, have only been recently emphasized~\cite{elmoataz2008nonlocal,couprie2012dual}. In this work, we show that total variation regularization can also be useful in the context of consensus problems.


Consensus algorithms designate a class of distributed methods allowing a set of connected agents\,/\,nodes to find an agreement on a some global parameter
value \cite{bertsekas:1997}.
The latter parameter is often defined as a minimizer of a global objective function defined as the sum of some local regret functions held by the
 agents 
\cite{lop-sayed-asap06,ram:nedic:veeravalli:jota-2010,bianchi:jakubo:2011}. As we shall see below, an important special case is obtained when the aim is to compute the average over the network of some local values held by the agents. The latter scenario will be refered to as the {\sl average consensus} case. It has been well-studied in the literature \cite{bertsekas:1997,kempe:dobra:gehrke:focs-2003,boyd:2006}. The most widespread approach to achieve average consensus is through iteration of linear operations mimicking the behaviour of heat equation \cite{boyd:2006}: at each round, nodes average the values in their neighborhood (including themselves). Similarly, in the more general framework of distributed  optimization, many algorithms have been proposed: some of them are based on distributed (sub)gradient approaches \cite{ram:nedic:veeravalli:jota-2010,bianchi:jakubo:2011,jakovetic2011fast,duchi2012dual} while others use splitting methods such as the Alternating Direction Method of Multipliers (ADMM) (see \cite{schizas:2008,Boyd2011,iutzeler2013asynchronous} and references therein). Under certain hypotheses, such approaches can be shown to converge to a state where each node in the network eventually has the same value -- the sought parameter.

However,  most  of  these works share a common view of the network: all agents show good will. They do not, for instance, deliberately introduce some false value inside the network, or refuse to update their value. There are a few recent work raising the problem of misbehaving agents in the gossip process \cite{eksin:2011,acemoglu:2011}, see also~\cite{mawlood2006issues} for a general perspective. In such scenarios, standard consensus algorithms not only fail, but can be driven arbitrarily far away from the sought consensus~\cite{walid2012robust}. A first approach to increase consensus robustness in unreliable networks is to detect misbehaving agents, identify them and finally exclude them from the network. Of course, cleaning the network beforehand is certainly beneficial whenever feasible, however
misbehaving agents are not necessarily detectable and even if they are, may be detectable only by using involved and computationally expensive algorithms. We refer to the recent works \cite{pasqualetti2012consensus,guo2012distributed}. An alternative is to design simple algorithms that naturally show good robustness properties.
For instance, the Authors of~\cite{lin2008distributed} study a continuous-time dynamical system allowing agents to track an agreement in the presence of external perturbation. The Authors of~\cite{zhang2012robustness} introduce a scheme in which each agent combines her/his current estimate with all but the extremal values received from her/his neighborhood. In the above works, it is worth mentioning that the objective is to ensure an (approximate) agreement between normal agents, irrespective to the value of this agreement. In this paper, our main interest is to build a robust consensus algorithm allowing the agents to find an agreement on a {\sl sought} parameter value.

{\bf Contribution.} Our contributions are the following.
\modif{
First, our definition of total variation on a graph is distinct from previous works~\cite{elmoataz2008nonlocal,couprie2012dual}. The distinction is that we use what is known as the \emph{anisotropic} total variation in the context of images and meshes. This subtle distinction has important consequences as it allows simple distributed algorithms.

Second, we cast the problem of robust distributed optimization over a network as an inverse problem using total variation regularization. To the best of our knowledge, this is a new
usecase for total variation regularization. Loosely speaking, this viewpoint amounts to think of consensus in a network as an extreme denoising process, where all the agents share the same value. In the context of image processing, it would amount to denoise until the image becomes totally grey.

Third, using our TV framework, we provide verifiable sufficient regularity conditions under which the minimizers of the relaxed problem coincide with the sought minimizers of the initial optimization problem. As a sanity check for the robustness of our algorithms, we analyze the convergence of our algorithms in the presence of stubborn agents that permanently introduce some false value in the network. We prove that unlike traditional approaches, our algorithms ensure that the estimates cannot be driven arbitrarily far away from the sought consensus.

Last, we provide  two iterative distributed algorithms which are shown to converge to the minimizers of the relaxed problem. Experimentally, we observe good convergence properties for the second (ADMM-based) algorithm.
}

The paper is organized as follows.  Section~\ref{sec:fmw} introduces the problem and the notations.
Section~\ref{sec:tvbasics} provides preliminary material on discrete Total Variation.
Section~\ref{sec:minimizers}  is devoted to the study of  the  minimizers of a relaxed distributed optimization problem .
Algorithms are proposed in Section~\ref{sec:algo}.
In Section~\ref{sec:stubborn}, we analyze the convergence of our algorithms in a scenario where stubborn agents are present.
Section~\ref{sec:simu} presents the numerical results.
Some of the proofs are given in appendixes.

\section{Distributed Optimization}
\label{sec:fmw}

\subsection{The Problem}

Consider a network of agents represented by an undirected  graph $G=(V,E)$ where $V$ is a finite set of agents and
$\{v,w\}$ belongs to $E$  if and only if agent $v$ and agent $w$ are able to communicate. We also use the notation
 $v\sim w$  when $\{v,w\}$ belongs to $E$.   We denote by $d(v)$ the degree of a vertex $v$ {\sl i.e.}, the number of neighbors in $G$.

We investigate the following optimization problem:
\begin{equation}
\label{eq:pb}
\inf_{x\in\bR} \sum_{v\in V} f_v(x)
\end{equation}
where $f_v:\bR\to \bR$ is a function which can be interpreted as the regret of agent $v$ when the network lies in a state $x$.
Merely for notational convenience, this paper is restricted to the case where parameter $x$ is real. Generalization to the case where $x$ belongs to an arbitrary Euclidean space is however straightforward.
We assume the following.
\begin{assumption} $ $\\
\modif{(a)}  For any $v\in V$, $f_v$ is a convex function.\\
\modif{(b)} The infimum of~(\ref{eq:pb}) is attained at some point~$x\in{\mathbb R}$.
\label{as:f}
\end{assumption}
\begin{example}
  We shall pay a special attention to the following particular case, which we shall refer to as the {\sl Average Consensus (AC)} case:
  \begin{equation}
    \label{eq:ACcase}
    \text{(AC) }\ \ f_v(x) = \frac 12 (x-x_0(v))^2
  \end{equation}
  where $x_0(v)$ represents some initial value held by agent $v$. In that case, problem~(\ref{eq:pb}) is equivalent to the distributed computation of the average $\overline x_0 = (1/{|V|})\sum_v x_0(v)$ where $|V|$ is the cardinal of~$V$.
\end{example}
\begin{example}
 A second special case of interest will be refered to as the {\sl Median Consensus (MC)} case:
 \begin{equation}
 \label{eq:MCcase}
 f_v(x) = |x-x_0(v)|\ .
 \end{equation}
 In this scenario, solving problem~(\ref{eq:pb}) is equivalent to searching for
 the median of sequence $(x_0(v))_{v\in V}$.
\end{example}

Each agent $v$ is supposed to hold some value $x_n(v)$ at each time $n\in \bN$. The aim of this paper is to introduce and analyze distributed algorithms which, under some assumptions, drive all sequences $(x_n(v))_{v\in V}$ to a common minimizer of (\ref{eq:pb}) as $n$ tends to infinity. Moreover, the proposed algorithms should be robust to the presence of misbehaving agents. By robust, we mean that the final estimate of regular (well-behaved) agents should remain in an acceptable vicinity of the sought consensus even in the case when other agents permanently introduce some false value in the network.

\subsection{Network Model}

Throughout this paper, we assume a synchronous network where a global clock allows the agents to communicate with each other at each clock tick.
In this paper, we refer to a {\sl distributed algorithm} as an iteration of the form:
\begin{equation}
\label{eq:regular}
x_{n+1}(v) = h_{n,v}\left((x_{k}(v),x_k(w) : w\sim v, 0\leq k\leq n)\right)
\end{equation}
for some specified functions $h_{n,v}$.
Nevertheless, we shall sometimes assume that some subset $S\subset V$ of agents do not follow the specified update rule~(\ref{eq:regular}).
Such agents will be called {\sl irregular}.
An irregular agent $v\in S$  is called  {\sl stubborn} if for any $n\geq 0$,
\begin{equation}
  \label{eq:stubborn}
  x_{n}(v) = x_0(v) \,.
\end{equation}
We denote by $S$ the set of irregular agents and by $R=V\backslash S$ the set of {\sl regular} agents.

\subsection{Variational Framework}
\label{sec:variational}


Consider replacing problem~(\ref{eq:pb}) with:
\begin{equation}
  \label{eq:reg}
  \min_{x\in\bR^V} \sum_{v\in V} f_v(x(v)) +  U(x)
\end{equation}
where $U:\bR^V\to \bR$ is a convex regularization penalizing the functions $x\in\bR^V$ that are away from the consensus space $\cC$.
There are several ways to choose $U$. The most immediate one is
$U = \iota_\cC$ as the indicator function of $\cC$ defined by  $\iota_\cC(x)=0$ if $x \in C$ and  $\iota_\cC(x)= + \infty$ otherwise.
In that case, problem~(\ref{eq:reg})
is equivalent to problem~(\ref{eq:pb}). From an intuitive point of view, setting $U = \iota_\cC$
means that consensus must be achieved {\sl at any price}. However, in the presence of irregular agents, it
is sometimes beneficial to break the diktat of consensus, in order to allow regular agents to possibly disagree with irregular ones.
Of course, for  $U \neq  \iota_\cC$, it can no longer be expected that the minimizers of (\ref{eq:reg}) coincide
in all generality with those of~(\ref{eq:pb}).
This can be seen as the  price to pay for an increased robustness. Nevertheless, we propose a way to select $U$ such that
the minimizers of (\ref{eq:reg}) coincide
with those of~(\ref{eq:pb}) {\sl at least for a certain class of functions}~$(f_v)_{v\in V}$.  We will focus  on  functions $U(x)=  \lambda\,\|x\|_{\tv}\ $ where $\lambda>0$ is a parameter to be specified and $\|x\|_{\tv}\  = \sum\limits_{ \{v,w\} \in E} |x(v) -x(w)|$.

In the sequel, we consider the following optimization problem:
\begin{equation}
  \label{eq:regTV}
  \min_{x\in\bR^V} \sum_{v\in V} f_v(x(v)) +  \lambda\,\|x\|_{\tv}\ .
\end{equation}

Intuitively, when $\lambda$ is large enough, the regularity term is dominant and the minimizer of~(\ref{eq:regTV}) is forced to the consensus subspace.
As shown in next sections,  this type of regularization functions allows some robustness against irregular agents   and leads to consensus when all agents are regular under some simple conditions.

In the AC problem, functions $f_v$ are given by~(\ref{eq:ACcase})
and the problem~(\ref{eq:regTV}) reduces to:
\begin{equation}
  \label{eq:rof}
  \min_{x\in\bR^V} \frac 12\|x-x_0\|_2^2+  \lambda\,\|x\|_{\tv}\ .
\end{equation}
In the context of image processing, the particular objective function~(\ref{eq:rof}) is referred to as the ROF (Rudin-Osher-Fatemi) energy~\cite{rudin1992nonlinear}.
We will refer to the general objective function in~(\ref{eq:regTV}) as a regularized energy, and to the minimizers of~(\ref{eq:regTV}) as regularized minimizers.

Our aim is threefolds: {\sl i)} to prove that the minimizers of~(\ref{eq:pb}) coincide with the regularized minimizers at least for a specified class of functions $f_v$;
{\sl ii)} to propose distributed algorithms to find regularized minimizers, {\sl iii)} to quantify the robustness of the algorithms in the presence of irregular (stubborn) agents.

 Let us first start with some general properties  related to total variation functions in a graph context.

\section{Total Variation on Graphs}

\label{sec:tvbasics}

The statements of this section can be seen as analogues of standard real analysis results. They will be of extensive use in Section~\ref{sec:minimizers}. Moreover, we believe these results can be of interest from a general perspective.
All proofs of this section are provided in Appendix~\ref{app:tvbasics}.

\subsection{Notations}

Consider an undirected graph $G=(V,E)$ where $V$ is a set of vertices and $E\subset V^2$ is a set of  edges.
Sometimes we are going to need an (arbitrary) orientation to each edge. $\vec{E}$ denotes whatever compatible set of directed edges, in the sense that $(v,w)\in\vec{E}$ implies that $\{v,w\}\in E$ and $(w,v)\not\in\vec{E}$; reciprocally $\{v,w\}\in E$ implies either $(v,w)\in\vec{E}$ or $(w,v)\in\vec{E}$. Of course, our results will not depend on the particular orientation we choose.





For a given set $A$, vector space $\mathbb{R}^A$ denotes the set of functions $A\to\mathbb{R}$, it is endowed with its standard vector space structure and scalar product $\langle f,g \rangle_A = \sum_{v\in A}f(v)g(v)$. Subscript $A$ will be omitted when no confusion can occur. $\zero_A$ stands for the constant function $v\in A\mapsto 0$ and $\one_A$ stands for the constant function $v\in A\mapsto 1$. The set $\cC$ of functions which are proportional to $\one_V$ is called the {\sl consensus subspace}. The cardinal of a set $A$ is denoted $|A|$.
The average of $x\in\mathbb{R}^A$ is denoted $\bar x = \frac{1}{|A|}\sum_{v\in A}x(v)$.
Notation $\grad$ accounts for the linear operator $\grad:\mathbb{R}^V \to  \mathbb{R}^{\vec E}$ defined for any $x\in {\mathbb R}^V$ by
$$
\grad(x): (v,w)\mapsto x(w) - x(v)\,.
$$
For instance, $\grad\one_V = 0_{E}$.
Notation $\div$ accounts for the operator $\div: \mathbb{R}^{\vec E} \to  \mathbb{R}^{V}$ defined for any $\xi\in {\mathbb R}^{\vec E}$ by
$$
\div(\xi): v\mapsto\sum_{(v,w)\in{\vec E}}\xi(v,w)-\sum_{(w,v)\in{\vec E}}\xi(w,v)\,.
$$
The following identity (integration by parts) holds:
 \begin{equation}
 \langle \grad f,\xi \rangle_{\vec E} = - \langle f,\div\xi\rangle_V
 \end{equation}
 In standard graph terminology, $\div$ is referred to as the {\sl incidence matrix}.
Operator $L = (-\div)\cdot\grad$ is the graph Laplacian associated to $G$.
Operator $L$ is positive semi-definite.

\subsection{Dual space of $(\mathbb{R}_0^V,\|\cdot\|_{\tv})$}

Let us denote by $\mathbb{R}_0^V$ the set $\{x\in\mathbb{R}^V:\langle x,1_V\rangle = 0 \}$ of zero-mean functions over $V$. 
It is straightforward to check that function $x\in\mathbb{R}^V_0\mapsto \sum_{e\in{\vec E}}|\grad x|(e)$ is a semi-norm on $\mathbb{R}_0^V$ and a norm when $G$ is connected. It is denoted $\|\cdot\|_{\tv}$ throughout the paper.
Although operator $\grad$ depends on the orientation chosen for $E$, note that $\|\cdot\|_{\tv}$ does not.

The dual space $(\mathbb{R}_0^V)^*$ identified with $\mathbb{R}_0^V$ using the standard scalar product is equipped with the dual norm:
\begin{equation}
\|u\|_* = \max_{\|x\|_{\tv}\leq 1}\langle x,u\rangle\,.
\label{eq:meyer}
\end{equation}
We introduce the unit ball:
$$
B_* = \{u\,:\, \|u\|_*\leq 1\}\ .
$$
Another characterization of the dual norm is the following.
For a vector field $\xi\in\mathbb{R}^{{\vec E}}$, we denote by $\|\xi\|_\infty=\max\{|\xi(e)|: e\in {\vec E}\}$.
The following proposition provides a characterization of the dual norm. Its proof is adapted from \cite{adams:1975}.


\begin{proposition}
  \label{prop:meyer-ball}
If $G$ is a connected graph, the following equality holds true:
  \begin{equation}
  \|u\|_* = \inf\{\|\xi\|_\infty: u = \div\xi\}\,.
 \label{eq:dual}
\end{equation}
\end{proposition}

The following property is a consequence of a general fact about subdifferentials of support functions:
\begin{proposition}
  \label{prop:subgradient-norm}
  If $\partial\|x\|_{\tv}$ denotes the subdifferential of norm $\|\cdot\|_{\tv}$ at point $x$, one has:
\[
\partial \|x\|_{\tv} = \{u\in\mathbb{R}_0^V:\|u\|_*\leq 1, \langle u,x\rangle = \|x\|_{\tv}\}
\]
In particular, $\partial \|0\|_{\tv} = B_*$\,.
\end{proposition}

\subsection{Co-area Formula}

First remark that $\|\cdot\|_{\tv}$ can be extended into a semi-norm on $\mathbb{R}^V$ using the same definition:
\[
\|x\|_{\tv} = \sum_{e\in{\vec E}} |\grad x|(e)\;.
\]
Using this definition, one has $\|x+c 1_V\|_\tv = \|x\|_\tv$ for any $c\in\mathbb{R}$ and any $x\in\mathbb{R}^V$.
The perimeter $\per(S)$ of a subset $S\subset V$ is defined as
\[
\per(S) = \|1_S\|_{\tv}\;.
\]
The following lemma, also known in the context of real analysis as the \emph{coarea formula}, will be helpful to prove Proposition~\ref{prop:meyer-perim}.

\begin{lemma}
  For a function $x\in\mathbb{R}^V$, we denote by $\{x\geq\lambda\} = \{v\in V: x(v)\geq\lambda\}$ the upper-level set associated with level $\lambda$. The following equality holds true:
\[
\|x\|_{\tv} = \int_{-\infty}^{+\infty}\per(\{x\geq\lambda\})\d\lambda\,.
\]
\label{prop:coarea}
\end{lemma}
The following useful result can be seen as an extension of the immediate formula $\|u\|_* = \max_{x\in\mathbb{R}^V}{\langle u,x\rangle}/{\|x\|_\tv}\,.$

\begin{proposition}
\label{prop:meyer-perim}
  Assume $u$ is in $(\mathbb{R}_0^V,\|\cdot\|_*)$. Then, using the canonical embedding $\mathbb{R}_0^V\subset\mathbb{R}^V$ and the standard inner product $\langle\cdot,\cdot\rangle$ over $\mathbb{R}^V$,
the following equalities hold true:
\begin{eqnarray*}
\|u\|_* &=& \max_{\emptyset\subsetneq S\subsetneq V}\frac{\langle u,1_S\rangle}{\|1_S\|_{\tv}} \\
 &=& \max_{ {\emptyset\subsetneq S\subset V, |S| \leq |V|/2} \atop {G(S) \mbox{ is connected}}}\frac{|\langle u,1_S\rangle|}{\|1_S\|_{\tv}}\,.
\end{eqnarray*}
\end{proposition}

\subsection{Dual norm computation}
\label{subsec:dualnormcomp}

As will be made clear in Section \ref{sec:minimizers}, it is essential to have in practice an efficient algorithm for the computation of the dual norm. In that perspective, Proposition~\ref{prop:meyer-perim} helps. We now propose a strongly polynomial-time combinatorial algorithm to compute the dual norm of a vector.


From Proposition \ref{prop:meyer-perim}, we know that $\|u\|_*$ can be computed by enumeration of all subsets $A$ of size  at most $|V|/2$ inducing a connected subgraph. The number of such subsets is polynomially bounded for some classes of graphs (e.g.,  paths and   cycles). However, in the general case, their number might not be polynomial. Another way to compute  $\|u\|_*$ consists in using either \eqref{eq:meyer} or \eqref{eq:dual}. Observe that \eqref{eq:meyer} or \eqref{eq:dual} are linear programs that can be solved in polynomial time using any standard linear programming algorithm.  In fact, \eqref{eq:dual} is simply the dual program of \eqref{eq:meyer}. Even if linear programming algorithms are very efficient, we will describe a strongly polynomial-time combinatorial algorithm to compute the dual norm of a vector, that is both practical and simple.

 \begin{breakbox}
   \noindent {\bf Algorithm~0}\label{algorithm:0}
 \begin{itemize}
 \item Select any subset $ \emptyset\subsetneq A_0 \subsetneq V$, let $\lambda_1 =  \frac{|\langle u,1_{A_0} \rangle|}{\per(A_0)}$ and $i=1$.

 \item Repeat
 \begin{itemize}
 \item Let $A_i =  \arg\max_{A \subset V} \langle u, 1_{A} \rangle  - \lambda_i \per(A)$

 \item Let $\lambda_{i+1} = \frac{\langle u,1_{A_i} \rangle}{\per(A_i)}$ and $i = i+1$
 \end{itemize}

 \item  Until $\lambda_{i} = \lambda_{i-1}$

 \end{itemize}

 \end{breakbox}

 Details related to the computation of  $\arg\max_{A \subset V} \langle u, 1_{A} \rangle  - \lambda_i \per(A)$ will be given later.

The following proposition is proved in Appendix~\ref{app:cdn}. 

 \begin{proposition}
 Algorithm~0 stops after at most $O(|E|)$ iterations. $\|u\|_*$  is given by the value of $\lambda_i$ at the last iteration.
 \label{prop:convergence0}
 \end{proposition}

\modif{In order to make {\bf Algorithm~0} practical, we still must specify how to solve the subproblem  $\max_{A \subset V} \langle u, 1_{A} \rangle  - \lambda \per(A)$. Let us now mention how this subproblem reduces to a standard \emph{max-flow/min-cut} problem \cite{korte:vygen:2012}}.

Recall that a network $(N,L,c)$ in graph theory sense is defined by a directed graph $(N,L)$ and a capacity  assignment $c_{(v,w)} \geq 0$ for any link  $(v,w) \in L$ .

Given the undirected graph $G$ and a vector $u \in \mathbb{R}_0^V$, we build a network $(N=V \cup \{s,t\}, L,c)$  as follows. For each edge $\{w,v\} \in E$ we create two directed edges $(w,v)$ and $(v,w)$ each of capacity $c_{(v,w)}=c_{(w,v)} = \lambda$.  In addition to all nodes of $V$, we add two other nodes: a source $s$ and a sink $t$.  Given any node $v$, if $u_v > 0$, we create a directed edge $(s,v)$ of capacity $u_v$, while an arc $(v,t)$ of capacity $|u_v|$ is added if $u_v <0$.    Thus, the set of links of the network is given by  $L=\{(v,w), \{v,w\} \in E\} \cup \{(s,v), u_v >0\} \cup \{(v,t), u_v < 0\}$.

Let  $A$ be  any subset of vertices of $V$ and  let $\delta^+ (A \cup \{s\})$ denote the set of directed edges having only their first extremity in  $S \cup \{s\}$.  $\delta^+ (A \cup \{s\})$ is generally called a cut. This cut separates $s$ and $t$ in the sense that  $s \in A \cup \{s\}$ while $t \notin A \cup \{s\}$.

  The capacity of this cut is defined as the sum of the capacities of the directed edges included in the cut. Let us denote it by $c(\delta^+ (A \cup \{s\}))$.   It is easy to see that $c(\delta^+ (A \cup \{s\}))$ is given by:
 \begin{eqnarray*}
 c(\delta^+ (A \cup \{s\}))  & =  &  \lambda \per(A) -   \sum\limits_{v \in A, u_v <0} u_v   +  \sum\limits_{v \in V \setminus A, u_v >0} u_v\nonumber \\
    & = &  \lambda  \per(A) -  \sum\limits_{v \in A} u_v + \sum\limits_{v \in A, u_v >0} u_v    +  \sum\limits_{v \in V \setminus A, u_v >0} u_v \nonumber \\
    & = & \lambda \per(A) - \langle u, 1_{A} \rangle +  \sum\limits_{v \in V, u_v >0} u_v \nonumber
 \end{eqnarray*}

 Observe that $c(\delta^+ (A \cup \{s\}))$ is the sum of the term  $\lambda \per(A) - \langle u, 1_{A} \rangle$ and a constant term not depending on $A$. Then, computing a minimum-capacity cut is clearly equivalent to finding a subset $A$ maximizing   $\max_{A \subset V} \langle u, 1_{A} \rangle  - \lambda \per(A)$.  A minimum-capacity cut can be computed using any maximum-flow/minimum-cut algorithm such as the Edmonds-Karp's algorithm, the Goldberg-Tarjan's algorithm or Orlin's Algorithm (see, e.g., \cite{korte:vygen:2012}).

Since each iteration of Algorithm~0 calls such a maximum-flow subroutine, and Proposition~\ref{prop:convergence0} asserts that there is at most $|E|$ iterations (see the appendix for a proof), the overall complexity of Algorithm~0 is consequently given by $|E|$ times the complexity of the maximum-flow algorithm (which depends on the algorithm used).

\section{Regularized minimizers}
\label{sec:minimizers}

Define function $F:\bR^V\to\bR$ by $F(x) = \sum_v f_v(x(v))$. For any $x\in \bR$, one has:
$$
\partial F(x1_V) = \left\{ u\in \bR^V : \forall v\in V,\, u(v)\in \partial f_v(x)\right\}\,.
$$
When all $f_v$'s are differentiable, note that $\partial F(x 1_V)$ is a singleton
$\{(f'_v(x))_{v\in V}\}$.
Recall that $B_*$ is the unit ball associated with the dual norm.

\begin{theorem}
Among the following statements, 1) 2) 3) are equivalent and  imply 4).
\begin{enumerate}
\item $x^\star1_V$ is a minimizer of~(\ref{eq:regTV}) ;
\item $\partial F(x^\star 1_V)\cap \lambda B_*$ is nonempty ;
\item There exists $u\in \partial F(x^\star 1_V)$ such that
$\sum_{v\in V}u(v)=0$ and for all $A\subset V$,
$$
\sum_{v\in A} u(v) \leq \lambda \per(A)\,.
$$
\item $x^\star$ is a minimizer of~(\ref{eq:pb}).
\end{enumerate}
\label{the:min}
\end{theorem}
\begin{proof}
[1) $\Leftrightarrow$ 2)] Note that $x^\star1_V$ is a minimizer of~$F+\lambda \|\cdot\|_\tv$ iff
$0\in \partial F(x^\star1_V) + \lambda \partial \|x^\star1_V\|_\tv$.
From Proposition \ref{prop:subgradient-norm}, $ \partial \|x^\star1_V\|_\tv=B_*$.
Therefore, 1) holds iff there exists $u\in \partial F(x^\star1_V)$ such that
$0\in u+\lambda B_*$. Otherwise stated, there exists $u\in \partial F(x^\star1_V)$
such that $u\in \lambda B_*$.
[2) $\Leftrightarrow$ 3)] is a consequence of Proposition~\ref{prop:meyer-perim}.
[3) $\Rightarrow$ 4)] As $\sum_v\partial f_v= \partial (\sum_v f_v)$, condition $\sum_{v\in V}u(v)=0$ implies that $0\in \partial (\sum_v f_v)(x^\star)$.
Thus,  $x^\star$ is a minimizer of $\sum_v f_v$.
\end{proof}

\modif{In order to have some insights about Theorem~\ref{the:min}, assume for instance that all $f_v$'s are differentiable.
Condition 2) of Theorem~\ref{the:min} can be simply rewritten as
\begin{equation}
  \label{eq:1}
  \|\nabla F(x^\star 1_V)\|_*\leq \lambda
\end{equation}
where $\nabla F(x^\star 1_V)$ is the gradient vector whose $v$th component is $f'_v(x^\star)$.
Now consider a solution $x^\star\in \bR$ to the initial problem~(\ref{eq:pb}). Theorem~\ref{the:min} states that whenever this solution satisfies~(\ref{eq:1}),
then $x^\star 1_V$ is also a solution to the relaxed problem~(\ref{eq:regTV}).}
\modif{Of course, condition~(\ref{eq:1}) remains a little abstract unless we have a way to verify the latter.
Algorithm~0 of Section~\ref{subsec:dualnormcomp} provides a practical method to compute the dual norm which can be used to verify condition~(\ref{eq:1}) in practice.
By statement 3) of Theorem~\ref{the:min}, condition~(\ref{eq:1}) is equivalent to:
\begin{equation}
\label{eq:condDeriv}
\forall A\subset V,\ \sum_{v\in A} f'_v(x^\star) \leq \lambda \per(A)\ .
\end{equation}
We thus have the following Corollary.
\begin{corollary}
  Assume that $f_v$ is differentiable for all $v\in V$. Let $x^\star$ be a minimizer of~(\ref{eq:pb}) satisfying condition~(\ref{eq:condDeriv}).
Then, $x^\star 1_V$ is a minimizer of~(\ref{eq:regTV}).
\end{corollary}
}
We now review the consequences of Theorem~\ref{the:min} regarding the Average Consensus and the Median Consensus problems described earlier in the introduction.

\subsection{Average Consensus Problem}

\begin{proposition}
The following statements are equivalent.
\begin{enumerate}
\item $\overline x_01_V$ is the unique minimizer of~(\ref{eq:rof}) ;
\item $x_0-\overline x_0\,1_V\in \lambda B_*$ ;
\item For all $A\subset V$,
$$
\left|\frac{\sum_{v\in A} x_0(v)}{|A|}-\bar x_0\right|\leq\lambda\frac{\per(A)}{|A|}
$$
\end{enumerate}
\label{prop:ac}
\end{proposition}
Proposition~\ref{prop:ac} quantifies how much a local average can fluctuate around $\bar x_0$ in order to preserve the sought equilibrium at $\bar x_0 1$: the larger the ratio $\per(A)/|A|$ the more it can fluctuate safely inside $A$. The heuristic behind this argument is that the ratio $\per(A)/|A|$ measures how well a given region $A$ is connected to the rest of the network, since $\per(A)$ is the number of edges connecting $S$ to its complementary set $V\backslash A$ and $|A|$ is a measure of its size. A large ratio $\per(A)/|A|$ amounts to say that relatively to its size, $A$ has a lot of connections to its outside. Being well connected, a subset $A$ is more able to make its member agree on $\bar x_0$ using its outside neighbors in $V\backslash A$; while a little connected part $A$ should already agree right from the start in order to hope a local consensus.

For $\lambda$ large enough, the critical value being $\|x_0-\bar x_0 1_V \|_*$, the minimizer of~(\ref{eq:rof}) is the sought equilibrium point. In other terms, there is a whole range of values for $\lambda$, namely the interval $[\|x_0 - \bar x_0 1_V \|_*,+\infty)$ for which the minimizer is {\sl exactly} $\bar x_01$.

Notice that if the data $x_0$ belong to a known bounded interval, then it is easy to compute an upper bound of  $\|x_0 - \bar x_0 1_V \|_*$ and to select $\lambda$ above this bound.
Checking this condition requires the computation of $\|x_0 - \bar x_0 1_V \|_*$. As far as computation of the dual norm is concerned, we refer to  Section \ref{subsec:dualnormcomp}.


 \subsection{Median Consensus Problem}
 \label{subsec:medcons}

 We denote by $\median(x_0)$ the set of minimizers of~(\ref{eq:pb}) when $f_v(x) = |x-x_0(v)|$ for all $v\in V$.
 It is straightforward to show that
 $$
 \begin{array}[h]{lll}
 \median(x_0) =&     \{x_0\circ\sigma(\frac{|V|+1}2)\} & \text{ if }|V|\text{ is odd}\\
 \median(x_0) =&    \left[x_0\circ\sigma(\frac{|V|}2)\,,\,x_0\circ\sigma(\frac{|V|}2+1)\right] & \text{ if }|V|\text{ is even.}
   \end{array}
 $$
 where $\sigma:\{1\cdots |V|\}\to V$ is any bijection such that
 $(x_0\circ \sigma) (1)\leq \cdots \leq (x_0\circ \sigma) (|V|)$.
 We use notation $\median(x_0)1_V$ to designate the set of functions $\{x1_V:x\in\median(x_0)\}$.
 We introduce the following sequence
 $$
 d = \left\{
   \begin{array}[h]{ll}
     (-1,\cdots,-1,0,1,\cdots 1) &  \text{ if }|V|\text{ is odd}\\
     (-1,\cdots,-1,1,\cdots 1) &  \text{ if }|V|\text{ is even.}
   \end{array}\right.
 $$
 We define by $\cU$ the set of functions $u\in\bR^V$
 for which there exists a bijection $\sigma:\{1\cdots |V|\}\to V$ such that
 $u\circ\sigma=d$. Otherwise stated, $\cU$ is the set of all permutations of sequence $d$.
 We set:
 $$
 \lambda_0 = \max\{ \|u\|_*\, :\, u\in \cU\}\ .
 $$

 \begin{proposition}
 \label{prop:MC}
   For any $\lambda > \lambda_0$, the set of \modif{regularized minimizers of the (MC) problem} is equal to  $\median(x_0)1_V$.
 \end{proposition}

\section{Proposed Algorithms}
\label{sec:algo}

\subsection{Subgradient Algorithm}

Note that function $\lambda\|x\|_{\tv}$ is non-differentiable.
Perhaps the most simple and natural approach is to use the subgradient algorithm
associated with problem~(\ref{eq:regTV}). This naturally yields the following distributed algorithm, where each node $v$ holds an estimate
$x_n(v)$ of the minimizer at time $n$ and combine it with the ones received from its neighbors.
\medskip

  \begin{breakbox}
    \noindent {\bf Algorithm~1}:
$$
x_{n+1}(v) = x_n(v) + \gamma_{n}\big[g_n(v) + \lambda \sum_{w\sim
  v}\sign\left(x_n(w)-x_n(v)\right)\big]
$$
\noindent where $g_n(v)\in -\partial f_v(x_n(v))$, typically $g_n(v) =
x_0(v)-x_n(v)$ in the (AC) case.
\end{breakbox}

Standard convex optimization arguments can be used to prove that function $x_n$ converges to a minimizer of~(\ref{eq:reg}) under the hypothesis of decreasing step size. The arguments being standard, the proof is omitted and we refer to \cite{bach2011optimization,boyd:lecture03} or references therein. A simple argument given in the appendix can be used to prove that the following property holds: $\forall n\geq 0, \bar x_n = \bar x_0$. The following theorem sums up the mentioned results.

\begin{assumption}
The following holds.
\begin{enumerate}
\item The step sizes satisfy $\gamma_n>0$ for all $n$, $\sum_n
  \gamma_n=+\infty$ and $\sum_n\gamma_n^2<\infty$.
\item There exists a constant $C>0$ such that for any $v$, any $x$ any $g\in \partial f_v(x)$,
 $|g|\leq C(1+|x|)$.
\end{enumerate}
\label{as:sousgrad}
\end{assumption}

\begin{theorem}
\label{the:CVsubg}
Consider that $R=V$ (all agents are regular).
Under Assumptions~\ref{as:f},~\ref{as:sousgrad}, sequence $x_n$ given by Algorithm~1 converges to the minimizers of~(\ref{eq:regTV}). Moreover, in the average consensus case (AC), the following property holds: $\forall n\geq 0, \bar x_n = \bar x_0$.
\end{theorem}
Theorem~\ref{the:CVsubg} is an immediate consequence of Proposition~\ref{prop:subg} given in Appendix~\ref{app:subg}.
This theorem is also a consequence of well known results on the subgradient algorithm \cite{shor1985minimization}.
However, for the sake of completeness, we provide a self-contained proof in this paper.

\subsection{Alternating Direction of Method of Multipliers (ADMM)}
\label{sec:admm}

It is a widely acknowledged fact that the subgradient method is slow in terms of convergence rate (see~\cite[Chap. 3.2.3]{nesterov2004introductory}).
Many alternatives do exist in order to speed up the convergence \cite{combettes-pesquet-book1-2011,bach2011optimization}.
Among these solutions, we propose an approach which can be seen as a special case
of ADMM. We refer to \cite{Boyd2011,iutzeler13,schizas:2008,erseghe2011fast} for other examples of applications of ADMM to distributed optimization.


Let us denote by $(V,\overset{\rightleftharpoons}{E})$ the directed graph such that:
$(v,w)\in \overset{\rightleftharpoons}{E}$ iff $v\sim w$.
That is, each pair $\{v,w\}$ of connected nodes yields two edges in $\overset{\rightleftharpoons}{E}$ (one from $v$ to $w$, the other from $w$ to $v$).
Problem~(\ref{eq:regTV}) is equivalent to
$$
\min_{(x,z)}\ \sum_{v\in V} f_v(x(v)) +  \lambda\,\sum_{\modif{\{v,w\}\,:\,v\sim w}} |z(v,w)-z(w,v)|
$$
where the minimum is taken w.r.t. $(x,z)\in\bR^V\times \bR^{\overset{\rightleftharpoons}{E}} $ such that
$z(v,w) = x(w)$ for any $(v,w)\in \overset{\rightleftharpoons}{E}$.
The augmented Lagrangian writes:
\begin{multline*}
  \cL(x,z;\eta) = \sum_{v\in V} f_v(x(v)) +
  \,\sum_{\modif{\{v,w\}\,:\,v\sim w}} \lambda|z(v,w)-z(w,v)| \\+
  \sum_{(v,w)\in \overset{\rightleftharpoons}{E}} T_\rho\left(\eta(v,w),z(v,w)-x(w)\right)
\end{multline*}
where we set $T_\rho(\alpha,\beta) = \alpha\beta+  \frac\rho2\beta^2$.
The ADMM consists in generating three sequences $(x_n,z_n,\eta_n)_{n\geq 0}$ recursively defined by
\begin{eqnarray*}
  x_{n+1} &=& \arg\min_{x\in\bR^V}\cL(x,z_n;\eta_n) \\
  z_{n+1} &=& \arg\min_{z\in\bR^{\overset{\rightleftharpoons}{E}}}\cL(x_{n+1},z;\eta_n)\\
\eta_{n+1}(v,w) &=& \eta_{n}(v,w)+\rho\left(z_{n+1}(v,w)-x_{n+1}(w)\right)
\end{eqnarray*}
for all $(v,w)\in\overset{\rightleftharpoons}{E}$.
\modif{In Appendix~\ref{app:admm}, we make ADMM explicit and prove that the update equation in $x_n$ is given by~{Algorithm~2} below.}

Denote by $\text{proj}_{[-\omega,\omega]}(x)$ the projection of $x$ onto $[-\omega,\omega]$
and by $\prox_{f,\rho}(x)=\arg\min_y f(y)+\frac \rho 2(y-x)^2$ the proximal operator associated with a real function $f$.
\medskip

\clearpage

\begin{breakbox}
  \noindent {\bf Algorithm~2}:

\noindent \modif{Each agent $v$ maintains the variables $x_n(v)$, $\bs\mu_n(v) =(\mu_n(w,v))_{w:w\sim v}$ for $n=0,1,\dots$.}

\noindent \modif{At time $n$, each agent $v$ sends $x_n(v)$ to  her/his neighborhood. Variable $\bs \mu_n(v)$ is kept private.}

  \noindent At  time $n$, each agent $v\in R$ receives $(x_n(w) : w\sim
  v)$ and makes the following updates:
$$
\mu_{n+1}(w,v) =
\text{proj}_{[-2\lambda/\rho,2\lambda/\rho]}\left(\mu_{n}(w,v)+x_{n}(w)-x_{n}(v)\right)
$$
for any $w\sim v$ in its neighborhood. Next,
$$
x_{n+1}(v) =
\prox_{f_v,\,\rho d(v)}\left(x_{n}(v)+\textstyle\frac 32
  \,\tilde\mu_{n+1}(v) - \frac 12\,\tilde\mu_{n}(v) \right)
$$
where we set $\tilde \mu_n(v) = \frac 1{d(v)}\sum_{w\sim
  v}\mu_n(w,v)$\,.
\end{breakbox}


Here, each regular agent $v$ not only maintains an estimate $x_n(v)$ of the minimizer, but also
holds in its memory one scalar $\mu_n(w,v)$ for any of its neighbors $w\sim v$.
We stress the fact that the values $\mu_n(w,v)$ are purely private in the sense that they are not exchanged by agents.
\modif{Otherwise stated, at time $n$,  an agent $v$ only shares her/his estimates $x_{n}(v)$ with her/his neighbors, the variable $\bs\mu_{n}(v)$ being private}.
In the (AC) case, operator $\prox_{f_v,\rho}$ has a simple expression.
In that case, the update equation in $x_n$ simplifies to:
$$
x_{n+1}(v) = \frac{x_0(v)+\rho d(v)\left(x_{n}(v)+\textstyle\frac 32 \,\tilde\mu_{n+1}(v)  - \frac 12\,\tilde\mu_{n}(v) \right)}{1+\rho d(v)}\,.
$$
As {Algorithm~2} can be seen as a special case of a standard ADMM, the following result follows directly from~\cite{combettes-pesquet-book1-2011}.
\begin{theorem}
\label{the:CVadmm}
  Assume that $R=V$ (all agents are regular). Under Assumption~\ref{as:f}, sequence $x_n$ given by Algorithm~2 converges to the minimizers of~(\ref{eq:regTV}). Moreover, in the average consensus case (AC), if the graph $G$ is assumed $d$-regular, {\sl i.e.} each node has the same degree $d$, the average is preserved over the network: $\forall n\geq 0, \bar x_n = \bar x_0$.
\end{theorem}

\section{Stubborn Agents}
\label{sec:stubborn}

One of the claims of this paper is that the above algorithms are attractive in order to provide robustness against misbehaving agents. This claim is motivated by the example below. Assume that some agents, called {\sl stubborn}, never change their state. The rationale behind this model is twofold: either these agents are malfunctioning, or they might want to deliberately pollute or influence the network.

For ease of interpretation, we will focus on the average consensus case.
We represent the state vector as $x_n = (x_n^R,x_n^S)$ where $x_n^R$ (resp. $x_n^S$) is the restriction of $x_n$ to regular agents (resp. stubborn agents). By definition, $x_n^S=x_0^S$ for any $n$.

\subsection{\modif{Failure of }Linear Gossip}

Assume that the state vector $ x_n$ is written in bloc form $ x_n = \begin{pmatrix} x^R_n\\  x^S\end{pmatrix}$.  The state vector is updated according to standard linear gossip scheme $ x_{n+1} =  W  x_n$ where $W$ is a square matrix.
We make the following natural assumptions about the matrix $W$.

\begin{assumption}[Linear Gossip Structure]
  \label{assum:stubgoss_rand}
  \mbox{}\\
  (a)  Matrix $ W$ has the following structure;
  \begin{equation}
       W = \begin{pmatrix}
         W^R &  W^S\\
         0 &  I
      \end{pmatrix}
    \end{equation} \\
   (b) Matrix $ W$ is a right stochastic matrix: its entries are in $[0,1]$ and   $ W  1_V =  1_V$. \\
  (c) Considering the directed edge
    structure $E'$ defined by: $(v,w)\in E'\Leftrightarrow
    W(v,w)>0$, there exists a directed path from each regular node $r$ to at least
    one stubborn node $s$.
  \end{assumption}

The bloc structure of matrix $ W$ follows from the constraint that {\sl stubborn} agents do not change their state over time. The last requirement is also very natural. If it were to be not fulfilled, there would exist regular agents that communicate in autarky and cannot be aware of {\sl stubborn} agents' opinions.

We now address the issue of convergence of such gossip algorithms in the
presence of {\sl stubborn} agents. The following result is proved in Appendix~\ref{app:deter-cvg}.

\begin{proposition}\label{prop:deterg-cvg}
  Under  Assumption~\ref{assum:stubgoss_rand}, algorithm $ x_{n+1} =  W x_n$ converges to
  \begin{equation}
 x_\infty = (I -  W^R)^{-1} W^S x^S
\end{equation}
\end{proposition}

Consensus is not necessarily reached since $ x_\infty$ is not proportional to $ 1_V$ as long as {\sl stubborn} agents disagree with each other. In addition, the limit state does not depend on the initial state, it only depends on the {\sl stubborn} agents state. In other terms, the {\sl stubborn} agents solely drive the network, initial opinions of regular agents is lost, even with one single {\sl stubborn} agent.

\subsection{ADMM and subgradient algorithm}


 As a straightforward extension of Theorems  \ref{the:CVsubg} and \ref{the:CVadmm}, it is not difficult to see that in the presence of stubborn agents, the sequence $x_n^R$ generated either by Algorithms~1 or~2 converges to the minimizers of the following {\sl perturbed} optimization problem
\begin{equation}
 \min_{ x\in \bR^R} \frac{1}{2} \| x - x^{R}_0 \|_2^2 +
\lambda\, \|  x\|_{\tv} +  \lambda \sum_{\substack{v \in R, w \in S\\ v\sim w}}
 |{ x(v) -  x_0(w)}|  \,.
\label{eq:13}
\end{equation}
where, here, $\|x\|_{\tv}$ is to be understood as the total variation of a function $x\in\bR^R$ on the subgraph $G(R)$ {\sl i.e.}, the restriction of $G$ to the set of regular agents. To ease the reading and with no risk of ambiguity, we still keep the same notations $\, \|  \,.\,\|_{\tv}$ and $\, \|  \,.\,\|_{*}$ to designate the TV and the dual norms associated $G(R)$.

\modif{A complete proof of robustness of our algorithm would require a closed-form expression of the minimizers of~(\ref{eq:13}). Unfortunately, such an explicit characterization is a difficult task. In the general case,
 solving analytically problem~(\ref{eq:13}) seems unfortunately out of reach.
Therefore, in order to validate the claim that our algorithms are indeed robust, we are left with two options. First, we should provide extensive numerical results that exhibits the robustness in practical scenarios. This is done in Section~\ref{sec:simu}. Second, we must prove robustness {\sl in some case study} for which the minimizers of (\ref{eq:13}) are tractable.
In the sequel, we characterize the minimizers in the following simplified scenario.}
\begin{scenario}
  Any stubborn agent is connected to all regular agents. In addition, there exists $a\in \bR$ such that $x_0(s)=a$ for any $s\in S$.
\label{hyp:stub}
\end{scenario}
Loosely speaking, one might think of Scenario~\ref{hyp:stub} as a worst-case situation in the sense that each stubborn agent directly disturbs {\sl all} regular agents. 

Let $\overline {x}_0^{R} = \frac 1{|R|}\sum_{v\in R} x_0^{R}(v)$. The following result is proved in Appendix~\ref{app:robTVGA}.
\begin{theorem}
\label{the:robTVGA}
  Assume that $\lambda \geq  \|x_0^{R} - \overline {x}_0^{R} 1_R \|_*$ and
  let
  \[
 x^* = \left\{
  \begin{array}[h]{ll}
  a  & \text{if } |\overline { x}_0^{R}  - a | \leq  \lambda   |S|   \\
    \overline { x}_0^{R}+ \lambda |S|& \text{if } \overline { x}_0^{R} + \lambda |S|  < a \\
    \overline {x}_0^{R}- \lambda |S|& \text{if } \overline {x}_0^{R} - \lambda |S|> a\,
 \end{array}\right.\,.
  \]
Then, in Scenario~\ref{hyp:stub}, $x^* 1_R$ is the unique minimizer of Problem~(\ref{eq:13}).
\end{theorem}

Observe that even if the common value $a$ of the {\sl stubborn} coalition  is very far from $\overline {x}_0^{R}$, we will reach a consensus   within a distance $\lambda |S|$ from $\overline {x}_0^{R}$. The same conclusion holds when  $a$ is already close to $\overline {x}_0^{R}$ (we still  reach a consensus within a distance $\lambda |S|$ from $\overline {x}_0^{R}$). The quantity $\lambda |S|$ can be interpreted as the robustness level of our algorithms in the sense of a maximum error margin. Therefore the proposed algorithm is unlike more standard gossip algorithms which can be driven arbitrarily far away from the sought consensus.
Note that a small $\lambda$ reduces the error margin. The tradeoff is related to the fact that the selection of a small $\lambda$ also reduces the set of functions $x_0^{R}$ satisfying the regularity condition $\|x_0^{R} - \overline {x}_0^{R} 1_R \|_*\leq \lambda$.

\section{Numerical Experiments}
\label{sec:simu}

In this section, the previous results are illustrated numerically. We first validate that, in a stubborn-free network, both Algorithm~1 and Algorithm~2 do converge to some $x^*$, minimizer of \eqref{eq:regTV}. Since eq.~\eqref{eq:regTV} involves abstract functions $f_v$, we consider two choices, namely $f_v(x) = (x-x_0(v))^2$ and $f_v(x)=|x-x_0(v)|$. Notice that in the average consensus case there is a single minimizer since the energy is strictly convex. We are able to predict the minimizers of~\eqref{eq:regTV} by computing $\|\partial F(x^*1_V)\|_*$ for $x^*=\bar x_0$ in the case $f_v(x) = (x-x_0(v))^2$ and $x^* \in \median x_0$ in the case $f_v(x) = \|x-x_0(v)\|$, using Algorithm~0 and then invoking Theorem~\ref{the:min}. We then introduce stubborn agents and validate Theorem~\ref{the:robTVGA} numerically checking that whenever $\lambda\geq\|x_0^R-\bar x_0^R\|_*$, regular agents do achieve consensus, settling the case of perturbed average consensus.

\subsection{Framework}

For the sake of simplicity we do not vary the underlying network in these experiments. The underlying network is the complete graph with $N=99$ agents (to have a unique median). In order to represent the data over the network, such as $v\in V\mapsto x(v)$, we choose an arbitrary order among the vertices, so that we can identify set $V$ with $\{1,\dots,N\}$. We then plot function $v\in\{1,\dots,N\}\mapsto x(v)$. Again for the sake of simplicity, whenever set $S$ is not empty -- {\sl i.e.} when there are some stubborn agents in the network -- we always assume there is a single stubborn agent indexed by $1$, namely: $S=\{1\}$.

The initial data is represented by circles in Figure~\ref{fig:initial_datum}. The average value of regular agents is approximately $\bar x_0 = 0.1271$ and $0.2817$ is a median.  In the case where there is a stubborn, the initial data is the same except for the first agent that is supposed stubborn.  Three possible values will be considered for the stubborn agent ($x_0(1) = 10$, $-10$ or $0.16$).
\begin{figure}[ht!]
  \[
  \begin{array}{cc}
	\includegraphics[width=.5\linewidth]{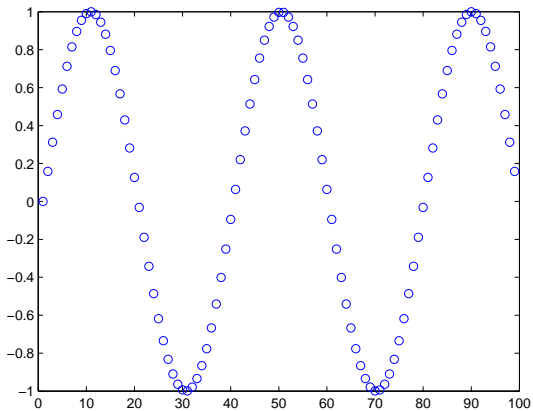}
  \end{array}
  \]
  \caption{Initial data (circles) as a function of agent index.
}
  \label{fig:initial_datum}
\end{figure}

\subsection{Convergence to the minimizers of eq.~\eqref{eq:regTV}}

We consider two cases: the average consensus where $f_v(x) = (x-x_0(v))^2$ and the median consensus problem $f_v(x) = |x-x_0(v)|$. In the average consensus case, the energy to minimize is strictly convex and admits a unique minimizer. In the median consensus case, convexity is not strict, and uniqueness does not necessarily hold.

For both cases, we provide two plots. The first plot shows function $n\mapsto \log\|J^\perp x_n\|$, where operator $J^\perp = I_N - J$, with $I_N$ the $n\times n$ identity matrix and $J$ the orthogonal projector on the consensus subspace (hence $J^\perp$ is the orthogonal projector on $1^\perp$). In other terms $\|J^\perp x_n\|$ measures how much disagreement there is left in the network. Please note the $\log$ in front, meaning that if such a function decrease linearly, it implies that $x_n$ goes to consensus exponentially fast. The second plot shows functions $n\mapsto \bar x_n$ -- except in the average consensus case -- since the graph is $N-1$-regular, $\bar x_n$ is kept constant along the iterations when using both Algorithms 1 and 2.

\subsubsection{Average Consensus}

In the regular case, we compute the target $\bar x_0 = 0.1271$. Next, we compute $\|x_0 - \bar x_0 1_v\|_* = 0.0130$ using Algorithm 0. We set $\lambda = 0.05 > \|x_0 - \bar x_0 1_v\|_*$. Figure~\ref{fig:algocompregavg} (left) shows function $n\mapsto\log\|J^\perp x_n\|$ for both Algorithms 1 and 2. Our results predict and the simulations confirm that $n\mapsto\bar x_n$ remain constant with $n$. It appears that Algorithm~2 goes to consensus exponentially fast.

\begin{figure}[ht!]
  \[
  \begin{array}{cc}
  \includegraphics[width=.5\linewidth]{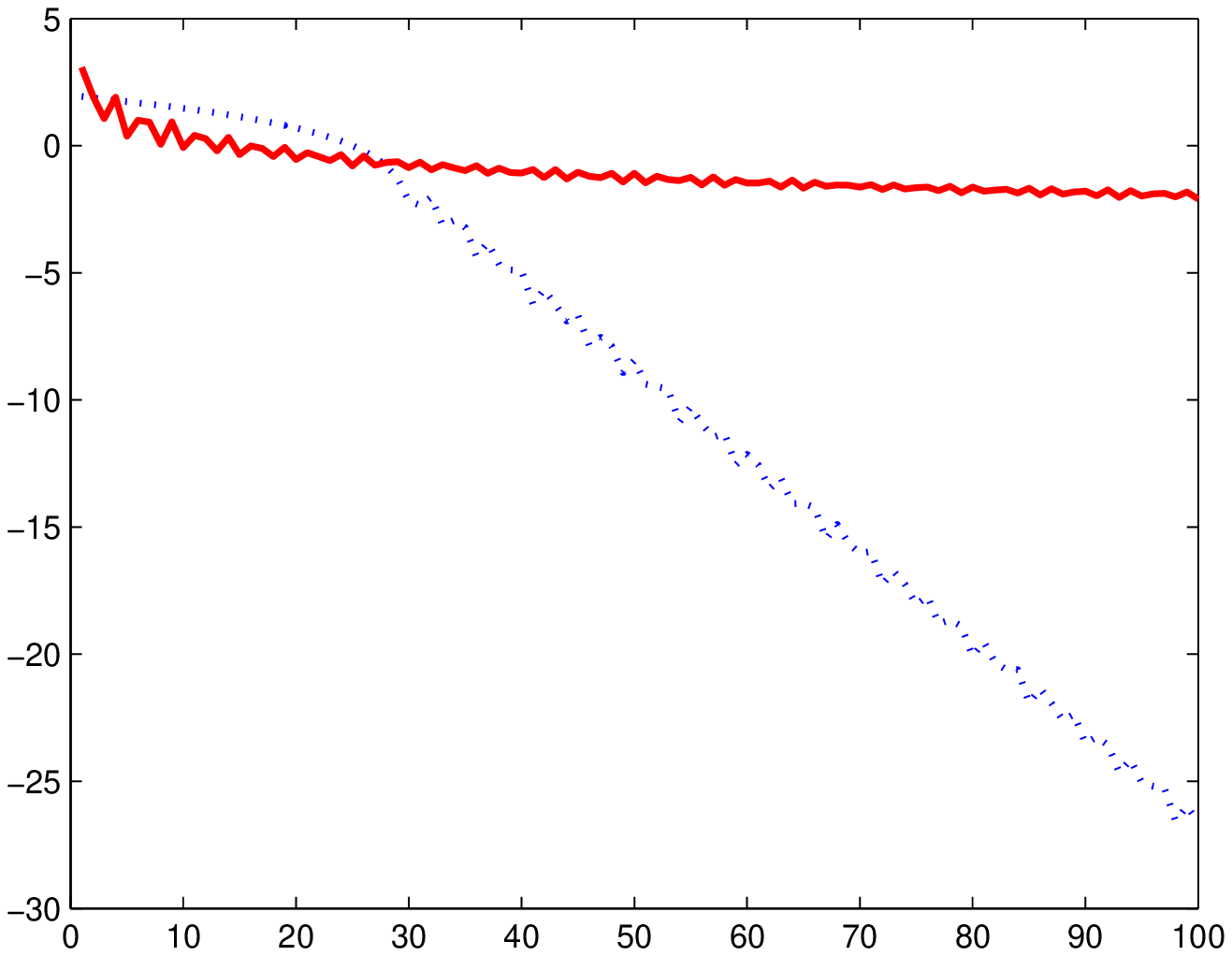}&
  \includegraphics[width=.5\linewidth]{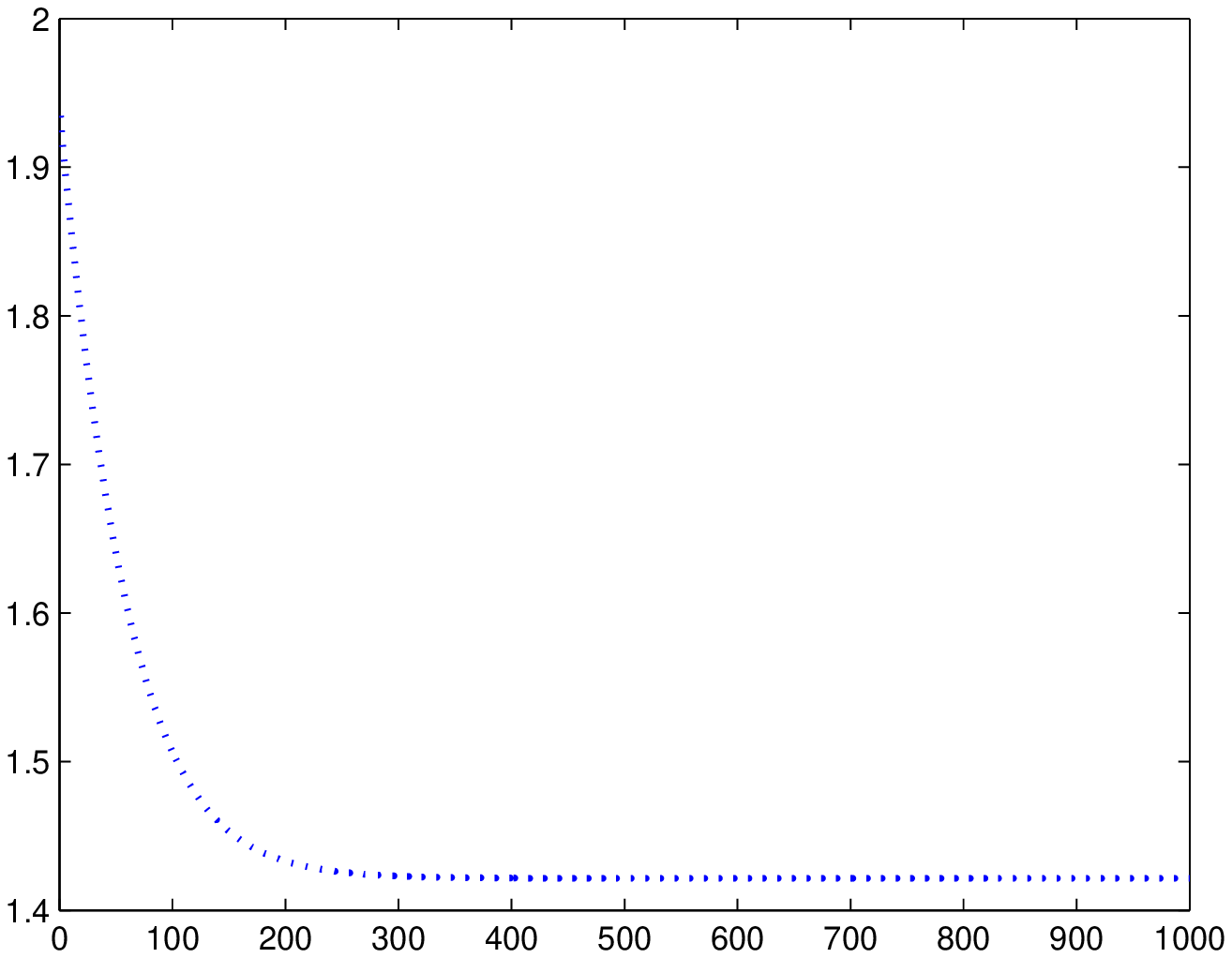}\\
  \end{array}
  \]
  \caption{Average consensus. Left: disagreement measure $n\to\log\|J^\perp x_n\|$ for Algorithms 1 and 2. Function $n\to\bar x_n$ is not represented since it is kept constant along the iterations. Note how faster Algorithm~2 converges to consensus. Right: Case where $\lambda < \|x_0 - \bar x_0 1_V\|_*$; consensus is not achieved and the quantity of disagreement, measured by $\log\|J^\perp x_n\|$ stabilizes after approximately $200$ iterations.}
  \label{fig:algocompregavg}
\end{figure}

On the opposite, we compare with the situation where $\lambda$ is subcritical, {\sl i.e.} less than $\|x_0 - \bar x_0 1_v\|_* = 0.0130$. We set $\lambda = 0.005$. Figure~\ref{fig:algocompregavg} (right) represents the curve $n\to\log\|J^\perp x_n\|$ for Algorithm~2 which appears the fastest. One can observe that consensus is not achieved, since $\log\|J^\perp x_n\|$ stays approximately constant after $200$ iterations instead of going down in a linear fashion as in the supercritical case, {\sl i.e.} when $\lambda > \|x_0 - \bar x_0 1_v\|_*$.

\subsubsection{Median Consensus}

In the regular case, we compute the target $\median x_0 = 0.2817$. For any non-empty subset $A \subset V$ in the complete graph it is easy to see that $\per(A )\geq N-1$. Hence, using Proposition~\ref{prop:meyer-perim} it is easy to see that $\lambda_0 \leq N/(2N-2)$. In order to satisfy $\lambda > \lambda_0$ we choose $\lambda=.5$.

Figure \ref{fig:algocompregmed} shows the two curves $n\mapsto\log\|J^\perp x_n\|$ and $n\mapsto\bar x_n$ for both Algorithms 1 and 2. One can notice that Algorithm~1 has not converged after $350$ iterations while Algorithm~2 has. One can also notice, comparing the slope breaks between the left and right figures that Algorithm~2 goes to consensus fast and then evolves at consensus towards the target value.

\begin{figure}[ht!]
  \[
  \begin{array}{cc}
  \includegraphics[width=.5\linewidth]{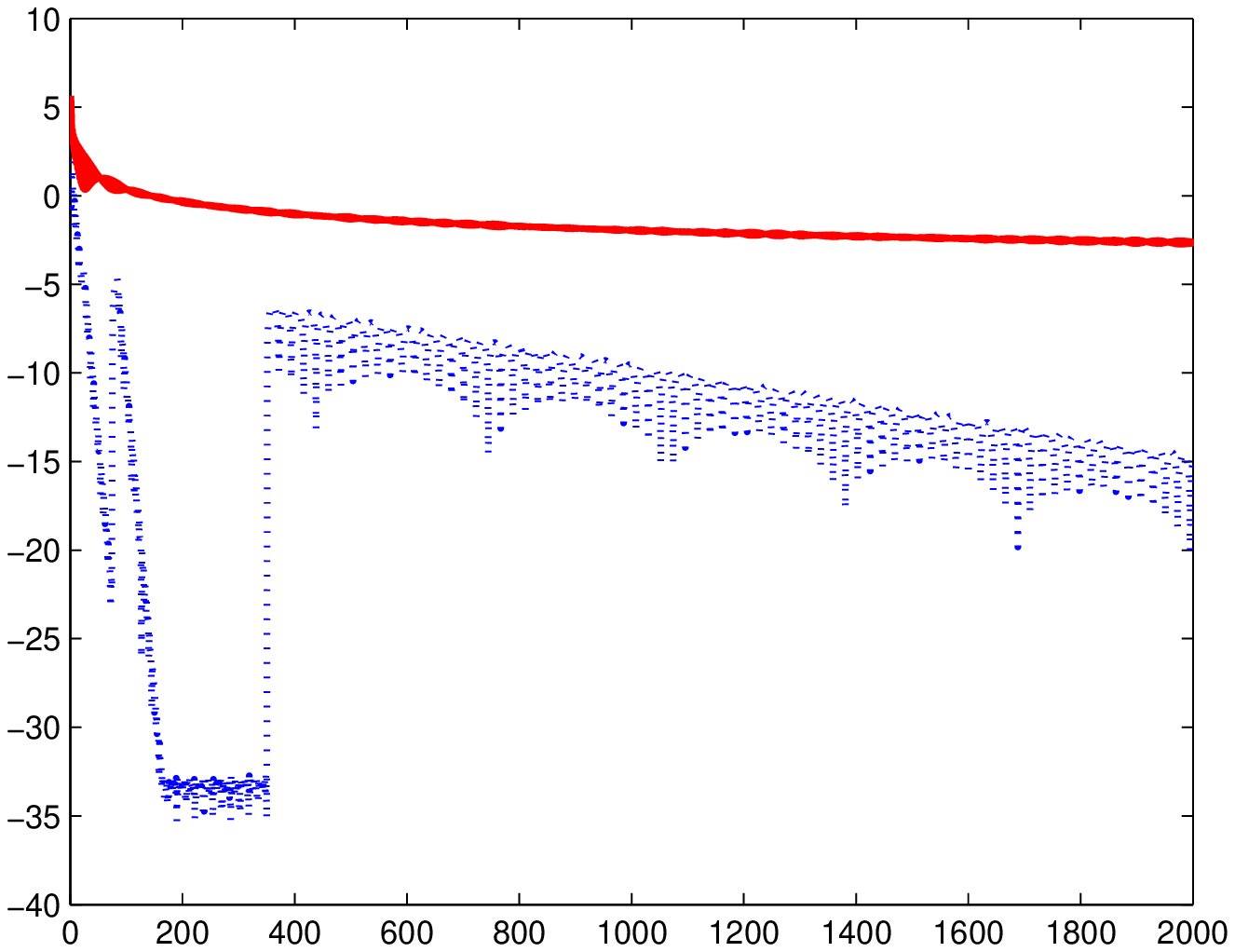}&
  \includegraphics[width=.5\linewidth]{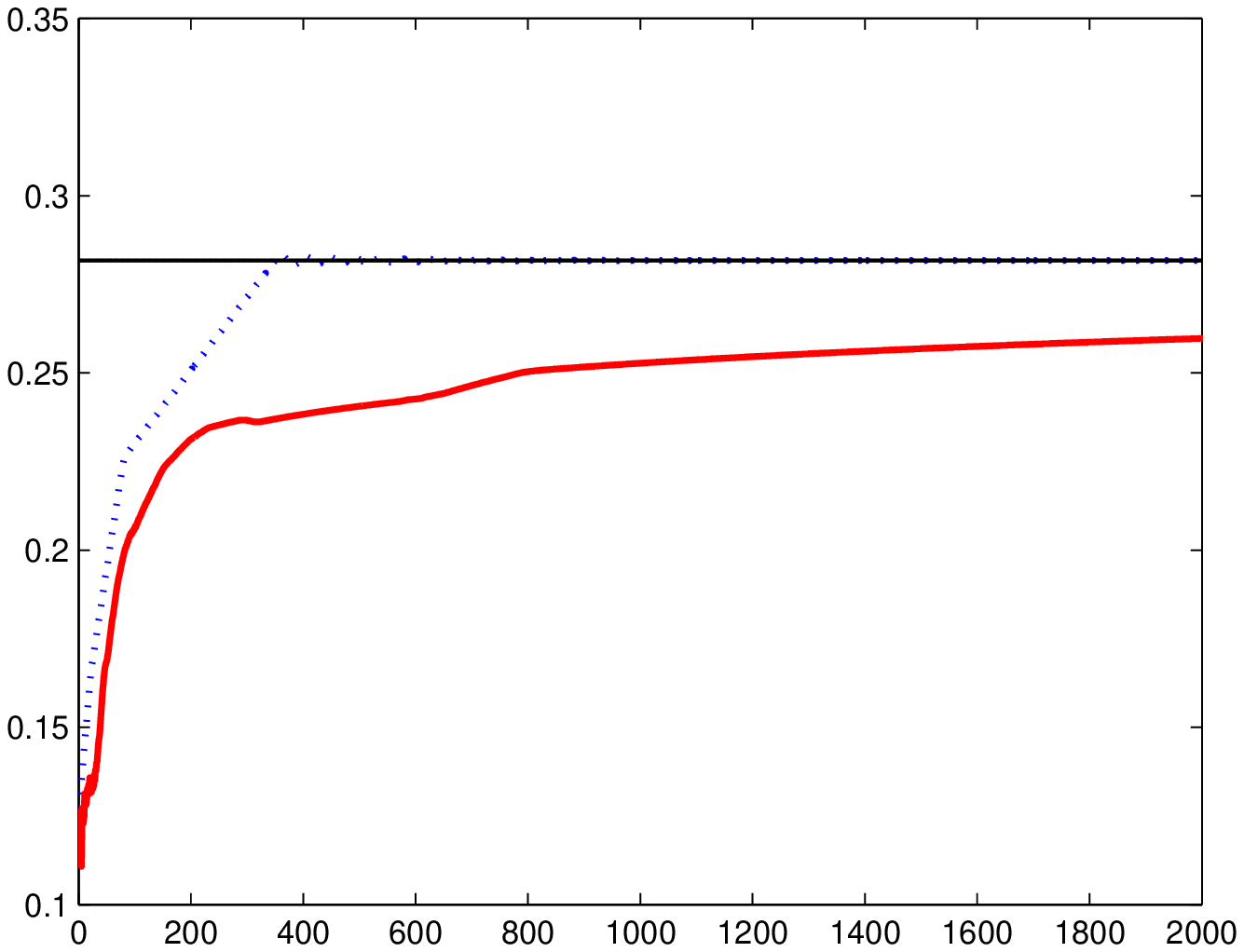}
  \end{array}
  \]
  \caption{Median consensus. Left: disagreement measure $n\to\log\|J^\perp x_n\|$ for Algorithms 1 and 2. Right: function $n\to\bar x_n$ for Algorithm 1 and 2, and a horizontal line corresponding to $\median(x_0)$. Observe an interesting feature of Algorithm~2: it reaches the target value at about iteration $350$ and then starts slightly oscillating around it; while Algorithm~1 does not converge after $2\,000$ iterations.}
  \label{fig:algocompregmed}
\end{figure}

\subsection{Behavior in the presence of stubborn agents}

We now investigate numerically the case when stubborn agents are introduced in the network, and more specifically the case where one single stubborn agent is introduced and the graph is the complete graph to validate Theorem~\ref{the:robTVGA}. We use the same initialization state as in the previous experiments, except that the first agent is assumed stubborn. The target value in this case is $\bar x_0^R = 0.1284$. Then we distinguish three cases: $x_0(1) = 10$, $x_0(1) = 0.16$ and $x_0(1) = -10$. When $\lambda = 0.05 \geq\|x_0^R - \bar x_0^R\|_* = 0.0133$ (as computed by Algorithm~0), there should be convergence to consensus to $\bar x_0^R + \lambda$ in the first case, $x_0(1)$ in the second case and $\bar x_0^R - \lambda$ in the third case; for both Algorithms. This is exactly what is found in the numerical experiments we performed as reported in Figure~\ref{fig:stub}; only experiments using Algorithm 2 are reported to avoid too much clutter.

\begin{figure}[ht!]
  \[
  \begin{array}{cc}
  \includegraphics[width=.5\linewidth]{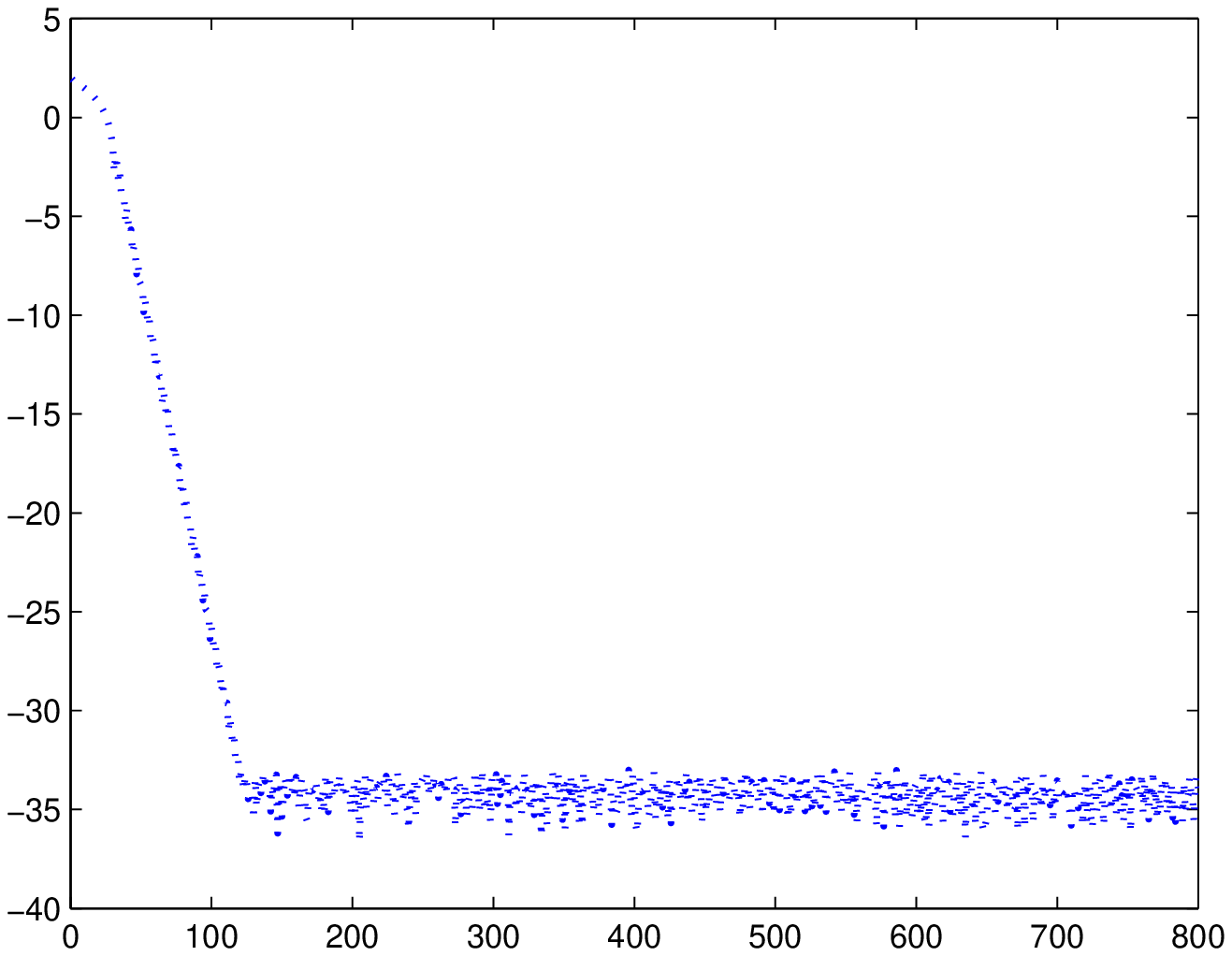}&
  \includegraphics[width=.5\linewidth]{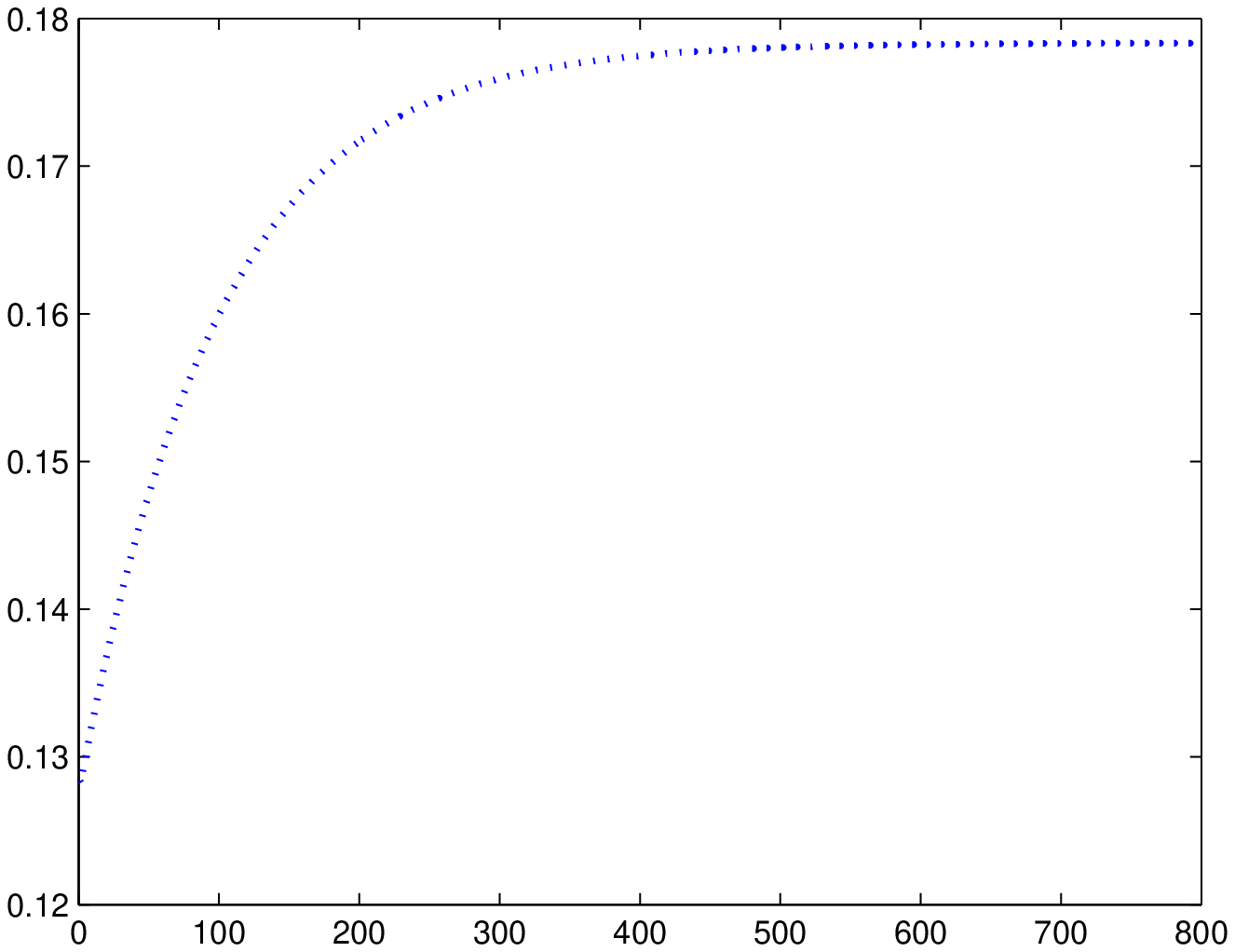}\\
  \includegraphics[width=.5\linewidth]{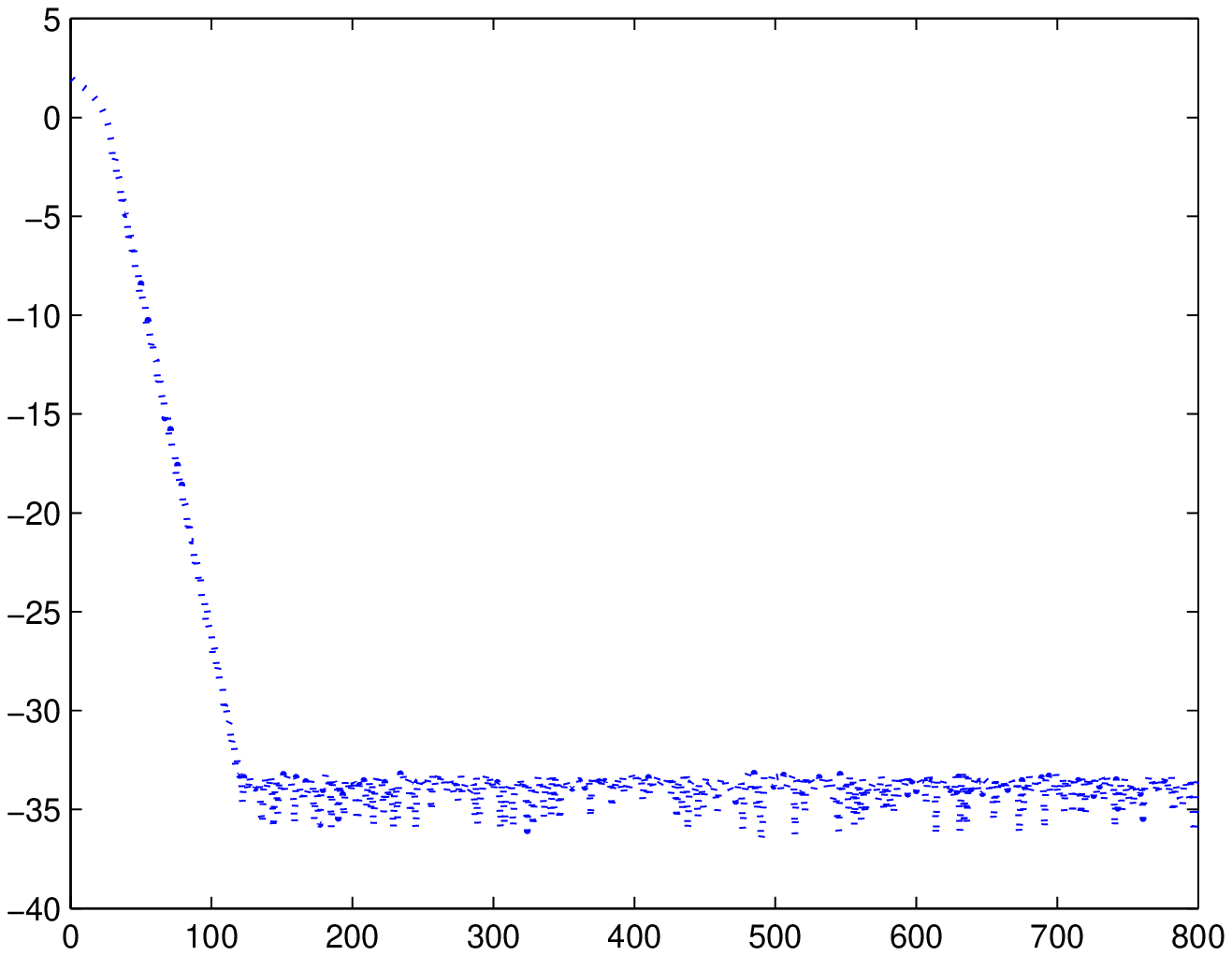}&
  \includegraphics[width=.5\linewidth]{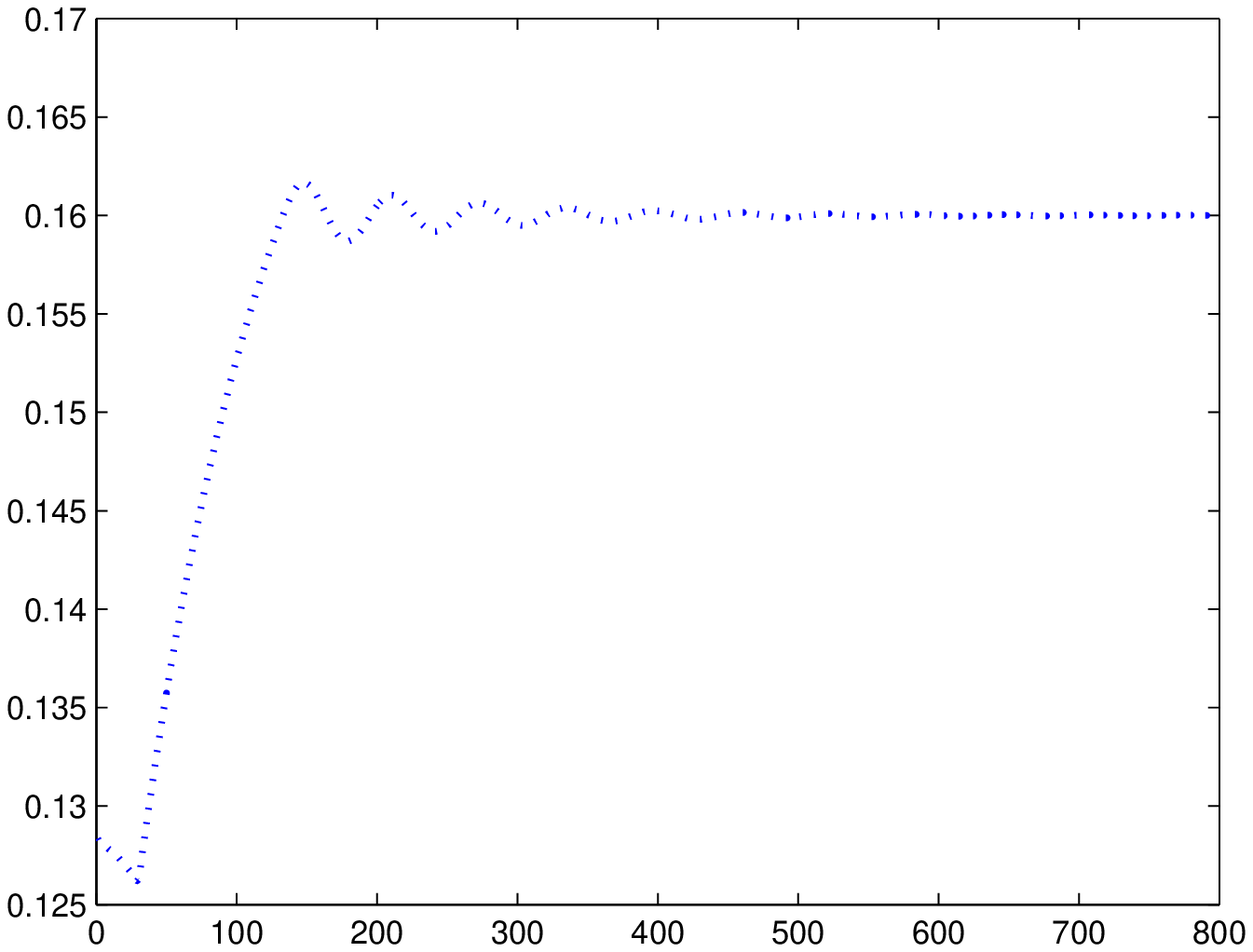}\\
  \includegraphics[width=.5\linewidth]{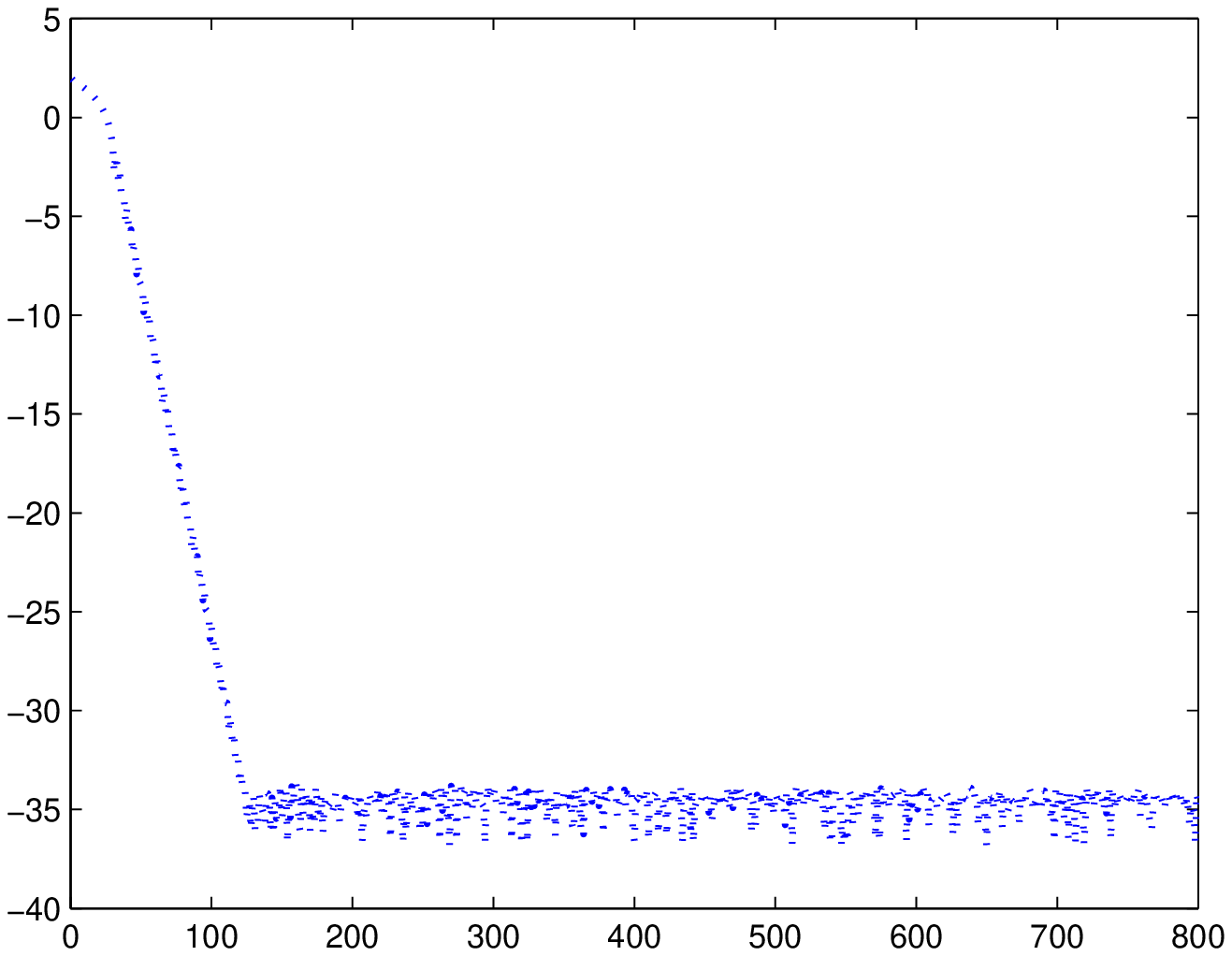}&
  \includegraphics[width=.5\linewidth]{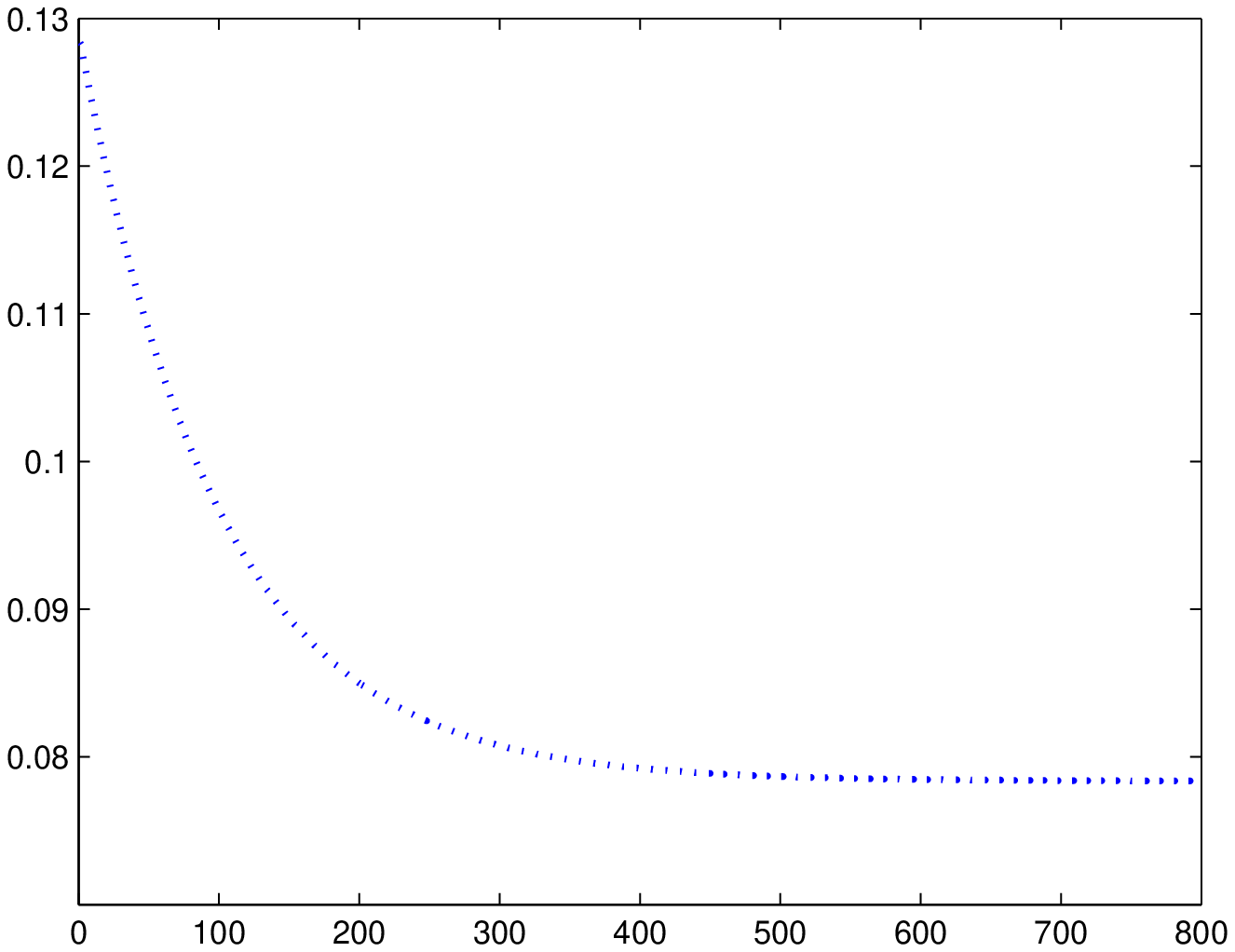}
  \end{array}
  \]
  \caption{With a stubborn agent. Left: disagreement measure $n\to\log\|J^\perp x_n\|$ for Algorithm 2. Right: function $n\to\bar x_n$ for Algorithm 2. First line stubborn agent has value $x_0(1) = 10$; $\bar x_n$ converges to $.1784$ which corresponds to $\bar x_0^R + \lambda$ up to precision $10^{-4}$. Mid line stubborn has value $x_0(1)=0.16$; again, well aligned with Theorem~\ref{the:robTVGA} since $\bar x_n$ converges to $0.16$. Third line, $x_0(1)= -10$; $\bar x_n$ converges to $0.0784$ that corresponds to $\bar x_0^R-\lambda$.}
  \label{fig:stub}
\end{figure}

\bibliographystyle{IEEEbib}
\bibliography{biblio}

\begin{thebibliography}{10}

\bibitem{rudin1992nonlinear}
Leonid~I Rudin, Stanley Osher, and Emad Fatemi,
\newblock ``Nonlinear total variation based noise removal algorithms,''
\newblock {\em Physica D: Nonlinear Phenomena}, vol. 60, no. 1, pp. 259--268,
  1992.

\bibitem{candes2005decoding}
Emmanuel~J Candes and Terence Tao,
\newblock ``Decoding by linear programming,''
\newblock {\em Information Theory, IEEE Transactions on}, vol. 51, no. 12, pp.
  4203--4215, 2005.

\bibitem{candes2006robust}
Emmanuel~J Cand{\`e}s, Justin Romberg, and Terence Tao,
\newblock ``Robust uncertainty principles: Exact signal reconstruction from
  highly incomplete frequency information,''
\newblock {\em IEEE Transactions on Information Theory}, vol. 52, no. 2, pp.
  489--509, 2006.

\bibitem{donoho2006compressed}
David~L Donoho,
\newblock ``Compressed sensing,''
\newblock {\em IEEE Transactions on Information Theory}, vol. 52, no. 4, pp.
  1289--1306, 2006.

\bibitem{ambrosio2000functions}
Luigi Ambrosio, Nicola Fusco, and Diego Pallara,
\newblock {\em Functions of bounded variation and free discontinuity problems},
\newblock Oxford Press, 2000.

\bibitem{elmoataz2008nonlocal}
Abderrahim Elmoataz, Olivier Lezoray, and S{\'e}bastien Bougleux,
\newblock ``Nonlocal discrete regularization on weighted graphs: a framework
  for image and manifold processing,''
\newblock {\em Image Processing, IEEE Transactions on}, vol. 17, no. 7, pp.
  1047--1060, 2008.

\bibitem{couprie2012dual}
Camille Couprie, Leo Grady, Laurent Najman, Jean-Christophe Pesquet, and Hugues
  Talbot,
\newblock ``Dual constrained tv-based regularization on graphs,''
\newblock {\em SIAM Journal on Imaging Science}, vol. 6, no. 3, pp. 1246--1273,
  2013.

\bibitem{bertsekas:1997}
D.~P. Bertsekas and J.~N. Tsitsiklis,
\newblock {\em {Parallel and Distributed Computation: Numerical Methods}},
\newblock Athena Scientific, 1997.

\bibitem{lop-sayed-asap06}
C.~Lopes and A.H. Sayed,
\newblock ``Distributed processing over adaptive networks,''
\newblock in {\em Adaptive Sensor Array Processing Workshop}, June 2006, pp.
  1--5.

\bibitem{ram:nedic:veeravalli:jota-2010}
S.S. Ram, A.~Nedic, and V.~Veeravalli,
\newblock ``Distributed stochastic subgradient projection algorithms for convex
  optimization,''
\newblock {\em Journal of Optimization Theory and Applications}, vol. 147, pp.
  516--545, 2010.

\bibitem{bianchi:jakubo:2011}
P.~Bianchi and J.~Jakubowicz,
\newblock ``Convergence of a multi-agent projected stochastic gradient
  algorithm for non-convex optimization,''
\newblock {\em IEEE Trans. on Automatic Control}, vol. 58, no. 2, 2013.

\bibitem{kempe:dobra:gehrke:focs-2003}
D.~Kempe, A.~Dobra, and J.~Gehrke,
\newblock ``Gossip-based computation of aggregate information,''
\newblock in {\em Foundations of Computer Science, Annual IEEE Symposium on}.
  2003, IEEE Computer Society.

\bibitem{boyd:2006}
S.~Boyd, A.~Ghosh, B.~Prabhakar, and D.~Shah,
\newblock ``{Randomized Gossip Algorithms},''
\newblock {\em IEEE Transactions on Inform. Theory}, vol. 52, no. 6, pp.
  2508--2530, 2006.

\bibitem{jakovetic2011fast}
Dusan Jakovetic, Joao Xavier, and Jose~MF Moura,
\newblock ``Fast distributed gradient methods,''
\newblock {\em arXiv preprint arXiv:1112.2972}, 2011.

\bibitem{duchi2012dual}
John~C Duchi, Alekh Agarwal, and Martin~J Wainwright,
\newblock ``Dual averaging for distributed optimization: convergence analysis
  and network scaling,''
\newblock {\em IEEE Transactions on Automatic Control}, vol. 57, no. 3, pp.
  592--606, 2012.

\bibitem{schizas:2008}
I.D. Schizas, A.~Ribeiro, and G.B. Giannakis,
\newblock ``{Consensus in ad hoc WSNs with noisy links-Part I: Distributed
  estimation of deterministic signals},''
\newblock {\em IEEE Transactions on Signal Processing}, vol. 56, no. 1, pp.
  350--364, 2008.

\bibitem{Boyd2011}
S.~Boyd, N.~Parikh, E.~Chu, B.~Peleato, and J.~Eckstein,
\newblock {\em {D}istributed optimization and statistical learning via the
  alternating direction method of multipliers},
\newblock Now Publishers Inc., 2011.

\bibitem{iutzeler2013asynchronous}
Franck Iutzeler, Pascal Bianchi, Philippe Ciblat, and Walid Hachem,
\newblock ``Asynchronous distributed optimization using a randomized
  alternating direction method of multipliers,''
\newblock {\em arXiv preprint arXiv:1303.2837}, 2013.

\bibitem{eksin:2011}
C.~Eksin and A.~Ribeiro,
\newblock ``Distributed network optimization with heuristic rational agents,''
\newblock {\em IEEE Trans. Signal Process,(submitted)}, 2011.

\bibitem{acemoglu:2011}
D.~Acemoglu, G.~Como, F.~Fagnani, and A.~Ozdaglar,
\newblock ``Opinion fluctuations and persistent disagreement in social
  networks,''
\newblock in {\em IEEE CDC 2011}, 2011, pp. 2347--2352.

\bibitem{mawlood2006issues}
A-R Mawlood-Yunis, Michael Weiss, and Nicola Santoro,
\newblock ``Issues for robust consensus building in p2p networks,''
\newblock in {\em On the Move to Meaningful Internet Systems 2006: OTM 2006
  Workshops}. Springer, 2006, pp. 1021--1027.

\bibitem{walid2012robust}
W.~Ben-ameur, P.~Bianchi, and J.~Jakubowicz,
\newblock ``Robust average consensus using total variation gossip algorithm,''
\newblock in {\em ValueTools' 12: 6th International Conference on Performance
  Evaluation Methodologies and Tools}, 2012.

\bibitem{pasqualetti2012consensus}
Fabio Pasqualetti, Antonio Bicchi, and Francesco Bullo,
\newblock ``Consensus computation in unreliable networks: A system theoretic
  approach,''
\newblock {\em IEEE Transactions on Automatic Control}, vol. 57, no. 1, pp.
  90--104, 2012.

\bibitem{guo2012distributed}
Meng Guo, Dimos~V Dimarogonas, and Karl~Henrik Johansson,
\newblock ``Distributed real-time fault detection and isolation for cooperative
  multi-agent systems,''
\newblock in {\em American Control Conference (ACC), 2012}. IEEE, 2012, pp.
  5270--5275.

\bibitem{lin2008distributed}
Peng Lin, Yingmin Jia, and Lin Li,
\newblock ``Distributed robust hinfinity consensus control in directed networks
  of agents with time-delay,''
\newblock {\em Systems \& Control Letters}, vol. 57, no. 8, pp. 643--653, 2008.

\bibitem{zhang2012robustness}
Haotian Zhang and Shreyas Sundaram,
\newblock ``Robustness of information diffusion algorithms to locally bounded
  adversaries,''
\newblock in {\em American Control Conference (ACC), 2012}. IEEE, 2012, pp.
  5855--5861.

\bibitem{adams:1975}
R.~A. Adams,
\newblock {\em Sobolev spaces},
\newblock Academic Press, 1975.

\bibitem{korte:vygen:2012}
B.~H. Korte and J.~Vygen,
\newblock {\em Combinatorial optimization}, vol.~21,
\newblock Springer, 2012.

\bibitem{bach2011optimization}
Francis Bach, Rodolphe Jenatton, Julien Mairal, and Guillaume Obozinski,
\newblock ``Optimization with sparsity-inducing penalties,''
\newblock {\em arXiv preprint arXiv:1108.0775}, 2011.

\bibitem{boyd:lecture03}
S.~Boyd, L.~Xiao, and A.~Mutapcic,
\newblock ``Subgradient methods,''
\newblock Tech. {R}ep., October 2003,
\newblock Lecture Notes.

\bibitem{shor1985minimization}
Naum~Zuselevich Shor,
\newblock {\em Minimization methods for non-differentiable functions},
\newblock Springer-Verlag, 1979.

\bibitem{nesterov2004introductory}
Y.~Nesterov,
\newblock {\em Introductory lectures on convex optimization: A basic course},
  vol.~87,
\newblock Springer, 2004.

\bibitem{combettes-pesquet-book1-2011}
Patrick~L. Combettes and Jean christophe Pesquet,
\newblock {\em Fixed-Point Algorithms for Inverse Problems in Science and
  Engineering}, chapter Proximal Splitting Methods in Signal Processing, pp.
  185--222,
\newblock Springer, 2011.

\bibitem{iutzeler13}
F.~Iutzeler, P.~Bianchi, Ph. Ciblat, and W.~Hachem,
\newblock ``Asynchronous distributed optimization using a randomized
  alternating direction method of multipliers,''
\newblock in {\em IEEE Conf. on Decision and Control}, 2013.

\bibitem{erseghe2011fast}
Tomaso Erseghe, Davide Zennaro, Emiliano Dall'Anese, and Lorenzo Vangelista,
\newblock ``Fast consensus by the alternating direction multipliers method,''
\newblock {\em Signal Processing, IEEE Transactions on}, vol. 59, no. 11, pp.
  5523--5537, 2011.

\end{thebibliography}

\appendices


\section{Proofs of Section~\ref{sec:tvbasics}}
\label{app:tvbasics}

\subsection{Proof of Proposition~\ref{prop:meyer-ball}}

\begin{proof}
First the set $\Xi_u = \{\xi\in\mathbb{R}^{{\vec E}}: u = \div\xi\}$ is not empty. Indeed, $\grad:\mathbb{R}_0^V\to\mathbb{R}^{{\vec E}}$ is into since $G$ is connected. Hence $\div:\mathbb{R}^{{\vec E}}\to\mathbb{R}^V_0$ is onto.

Assume $u = \div\xi$, then $\langle u,x\rangle = -\langle\xi,\grad x\rangle\leq\|\xi\|_\infty\sum_{e\in{\vec  E}}|\grad x|(e) = \|\xi\|_\infty\|x\|_{\tv}$. Hence $\|u\|_* \leq \|\xi\|_\infty$, and $\|u\|_* \leq \inf\{\|\xi\|_\infty:u = \div\xi\}$. Reciprocally, since $G$ is assumed connected, $\grad:\mathbb{R}_0^V\to\mathbb{R}^{{\vec E}}$ is into; let us denote by $R$ its range. Then $\grad:(\mathbb{R}_0^V,\|\cdot\|_{\tv})\to (R,\|\cdot\|_1)$ is an isometry. Moreover, using Hahn-Banach theorem (finite dimensional case), one can embed isometrically $(R,\|\cdot\|_1)^*$ into $(\mathbb{R}^{{\vec E}},\|\cdot\|_1)^*\simeq(\mathbb{R}^{{\vec E}},\|\cdot\|_\infty)$. Hence to $u\in(\mathbb{R}_0^V,\|\cdot\|_*)$ one can associate $\varphi\in\mathbb{R}^{{\vec E}}$ such that $\langle u,x\rangle = \langle\varphi,\grad x\rangle$ with $\|\varphi\|_\infty=\|u\|_*$. It suffices to take $\xi = -\varphi$ to have $u=\div\xi$ with $\|u\|_* = \|\xi\|_\infty$ to prove the reverse inequality.
\end{proof}

\subsection{Proof of Proposition~\ref{prop:subgradient-norm}}
\begin{proof}
Let $u$ be such that $\|u\|_* \leq 1$ and $\langle u,x\rangle = \|x\|_{\tv}$. Then, $\langle u,y-x\rangle + \|x\|_{\tv} = \langle u,y\rangle\leq \|u\|_*\|y\|_{\tv} = \|y\|_{\tv}$ which means that $u\in\partial\|x\|_{\tv}$.
Conversely, assume $u\in\partial\|x\|_{\tv}$ and $x_u$ is s.t. $\|x_u\|_{\tv} = 1$ and $\langle u,x_u\rangle = \|u\|_*$. Define $y_u = \|x\|_{\tv}x_u$; one has: $\|y_u\|_{\tv} - \|x\|_{\tv}\geq \langle u,y_u - x\rangle$, which gives $0\geq \|u\|_*\|x\|_{\tv} - \langle u,x\rangle$. By inequality $\|u\|_*\|x\|_{\tv} - \langle u,x\rangle\geq0$, one has $\langle u,x\rangle = \|u\|_*\|x\|_{\tv}$.
Moreover, as $u\in\partial\|x\|_{\tv}$ $\|2x\|_{\tv} - \|x\|_{\tv}\geq \langle u,x\rangle = \|u\|_*\|x\|_{\tv}$.  Consequently, f $x \neq 0$, then $\|u\|_* = 1$.

If $x = 0$, then writing $\langle u,x_u\rangle \leq   \|x_u\|_{\tv}$ directly leads to $\|u\|_* \leq 1$.
\end{proof}


\subsection{Proof of Lemma~\ref{prop:coarea}}

\begin{proof}
First notice that the integral is well defined since $\lambda\in\mathbb{R}\mapsto\per(\{x\geq\lambda\})$ has its support included in $[\min x,\max x]$ and is piecewise constant with finite values. For each edge $\{v,w\}\in E$, denote by $I_e = [x(v)\wedge x(w), x(v)\vee x(w)]\subset\mathbb{R}$. Now, it is easy to check that $\per(\{x\geq\lambda\}) = \sum_{e\in E}1_{I_e}(\lambda)$.
Hence,
\[
\int_{-\infty}^{+\infty}\per(\{x\geq\lambda\})\d\lambda = \sum_{e\in E}\int_{-\infty}^{+\infty}1_{I_e}(\lambda)\d\lambda = \sum_{e\in E}|I_e|
\]
where $|I|$ denotes the length $b-a$ of interval $I=[a,b]$. The rightmost term is equal to $\|x\|_{\tv}$, which completes the proof.
\end{proof}

\subsection{Proof of Proposition~\ref{prop:meyer-perim}}
\begin{proof}
  Since $\mathbb{R}_0^V$ is finite dimensional, there exists $x_u\in\mathbb{R}_0^V$ with $\|x_u\|_{\tv} = 1$ such that $\|u\|_* = \langle u,x_u\rangle$. Since $\langle u,1_V\rangle = 0$ one has $\langle u,\tilde x_u\rangle = \langle u,x_u\rangle$ with $\tilde x_u = x_u-(\min_vx_u(v))1_V$. Now, let us consider subsets of $V$ having the form $S_\mu = \{\tilde x_u\geq \mu\}$ for $\mu\in\mathbb{R}$. Notice that $S_\mu=V$ for $\mu\leq 0$ and $S_\mu=\emptyset$ for $\mu>M$ with $M>0$ large enough. Hence, the following integral is well defined:
\[
\int_{-\infty}^{+\infty}\langle u,1_{S_\mu}\rangle\d\mu
\]
And,
\[
\int_{-\infty}^{+\infty}\langle u,1_{S_\mu}\rangle\d\mu = \int_0^M\langle u,1_{S_\mu}\rangle\d\mu = \langle u,\int_0^M1_{S_\mu}\d\mu\rangle\;.
\]
where $\int_0^M1_{S_\mu}\d\mu$ denotes function $v\mapsto\int_0^M1_{S_\mu}(v)\d\mu$.
Moreover $\forall v\in V$,
\[
\int_0^M 1_{S_\mu}(v)\d\mu = \int_0^{+\infty}1_{\{\mu\leq \tilde x_u(v)\}}\d\mu=\tilde x_u(v)
\]
Hence
\[
\int_{-\infty}^{+\infty}\langle u,1_{S_\mu}\rangle\d\mu = \langle u,\tilde x_u\rangle\;.
\]
 By definition of $\|u\|_*$ and the fact that (i) $\langle u,1_V\rangle=0$, (ii) $\|x\|_{\tv} = \|x+c1_V\|_{\tv}$, one has $\langle u,1_{S_\mu}\rangle \leq \|u\|_*\|1_{S_\mu}\|_{\tv}$. Hence, function $\mu\mapsto \langle u,1_{S_\mu}\rangle - \|u\|_*\|1_{S_\mu}\|_{\tv}$ is nonpositive. Integrating and using the co-area formula, one gets, for almost every $\mu\in\mathbb{R}$:
\[
\langle u, 1_{S_\mu} \rangle = \|u\|_*\|1_{S_\mu}\|_\tv
\]
{\sl A fortiori}, the set of such $\mu$ is not empty, which concludes the proof \modif{of the first equality.

It remains to prove the second equality.
Using the fact that $u \in \mathbb{R}_0^V$, we get that $|\langle u, 1_{S} \rangle| = | \langle u, 1_{V \setminus S} \rangle|$.  Moreover, using the first equality of Proposition \ref{prop:meyer-perim} and  equality $\per(S) = \per(V \setminus S)$, we get that
$\|u\|_* = \max_{ {\emptyset\subsetneq S\subset V, |S| \leq |V|/2}}\frac{|\langle u,1_S\rangle|}{\|1_S\|_{\tv}}$. Let us consider a subset $S$  such that  $\|u\|_* = \frac{|\langle u,1_S\rangle|}{\|1_S\|_{\tv}}$ and assume that  $S = S_1 \cup S_2$  with $E(S_1,S_2) =\emptyset$ where $E(S_1,S_2)$ is the set of edges having one extremity in $S_1$ and one extremity in $S_2$.   Then the ratio  $\frac{|\langle u,1_S\rangle|}{\|1_S\|_{\tv}}$ can be written as
$ \frac{|   \langle u,1_{S_1} \rangle  + \langle u,1_{S_2}\rangle           |   }{ \|1_{S_1}\|_{\tv} + \|1_{S_2}\|_{\tv}  }$  which is less than or equal to $ \frac{|   \langle u,1_{S_1} \rangle|  + |\langle u,1_{S_2}\rangle           |   }{ \|1_{S_1}\|_{\tv} + \|1_{S_2}\|_{\tv}  }$.  The last ratio is clearly bounded by
$\max\left(       \frac{|   \langle u,1_{S_1} \rangle|   }{ \|1_{S_1}\|_{\tv}  }, \frac{|   \langle u,1_{S_2} \rangle|   }{ \|1_{S_2}\|_{\tv}  }   \right)$.  Since $S$ is maximizing the ratio $\frac{|\langle u,1_S\rangle|}{\|1_S\|_{\tv}}$, we should have  $\frac{|\langle u,1_S\rangle|}{\|1_S\|_{\tv}} =  \max\left(       \frac{|   \langle u,1_{S_1} \rangle|   }{ \|1_{S_1}\|_{\tv}  }, \frac{|   \langle u,1_{S_2} \rangle|   }{ \|1_{S_2}\|_{\tv}  }   \right)$. Consequently, to compute $\|u\|_*$, we can focus on subsets $S$ inducing a connected graph $G(S)$.
}

\end{proof}

\subsection{Proof of Proposition~\ref{prop:convergence0}}
\label{app:cdn}
\begin{proof}
First, observe that since ${\langle u,1_{S_{i-1}}\rangle} - \lambda_{i}  {\per (S_{i-1})} = 0 $, and $S_i  = \arg\max_{S \subset V} \langle u, 1_{S} \rangle  - \lambda_i \per(S)$, we necessarily have $\langle u, 1_{S_i} \rangle  - \lambda_i \per(S_i) \geq 0$. In other words, we have $\frac{ \langle u,1_{S_{i}}\rangle} {\per(S_i)  } \equiv \lambda_{i+1} \geq  \lambda_i   $.

Let us now assume that we obtained $\lambda_{i+1} =  \lambda_i$ for some $i$. This is equivalent to say that $\langle u, 1_{S_i} \rangle  - \lambda_i \per(S_i) = 0$. By definition of $S_i$, we should have
$\langle u, 1_{S} \rangle  - \lambda_i \per(S) \leq \langle u, 1_{S_i} \rangle  - \lambda_i \per(S_i)$ for each subset.  Consequently, $\frac{ \langle u,1_S \rangle} {\per(S)  } \leq  \lambda_{i}$ for each $S$ with equality for $S=S_{i-1}$. The convergence of the algorithm is then achieved.

To finish the proof, let us try to bound the number of iterations of the algorithm.
Assume that $\lambda_{i+1} > \lambda_{i}$. This implies that
$ \langle u, 1_{S_i} \rangle  - \lambda_i \per(S_i)  > 0 = \langle u, 1_{S_{i-1}} \rangle  - \lambda_i \per(S_{i-1})  $. Moreover, by definition of $S_{i-1}$, we should have
$\langle u, 1_{S_i} \rangle  - \lambda_{i-1} \per(S_i)  \leq \langle u, 1_{S_{i-1}} \rangle  - \lambda_{i-1} \per(S_{i-1})  $.  These two inequalities can be written as follows:
$  \lambda_i ( \per(S_i) -  \per(S_{i-1}))      < \langle u, 1_{S_i} \rangle  -   \langle u, 1_{S_{i-1}} \rangle   \leq   \lambda_{i-1} ( \per(S_i) -  \per(S_{i-1}))          $. This holds only if $\per(S_i) -  \per(S_{i-1}) <0$. Said another way, if $\lambda_{i+1} > \lambda_{i}$ then  $\per(S_i)$ decreases at least by one unit.  Since we know that $\per(S_i)$ is between $1$ and $|E|-1$, the  number of iterations is bounded by $|E|$.
\end{proof}

\section{Proofs of Section~\ref{sec:minimizers}}

\subsection{Proof of Theorem~\ref{the:min}}
\begin{proof}
[1) $\Leftrightarrow$ 2)] Note that $x^\star1_V$ is a minimizer of~$F+\lambda \|\cdot\|_\tv$ iff
$0\in \partial F(x^\star1_V) + \lambda \partial \|x^\star1_V\|_\tv$.
From Proposition \ref{prop:meyer-ball}, $ \partial \|x^\star1_V\|_\tv=B_*$.
Therefore, 1) holds iff there exists $u\in \partial F(x^\star1_V)$ such that
$0\in u+\lambda B_*$. Otherwise stated, there exists $u\in \partial F(x^\star1_V)$
such that $u\in \lambda B_*$.

[2) $\Leftrightarrow$ 3)] is a consequence of Proposition~\ref{prop:meyer-perim}.

[3) $\Rightarrow$ 4)] As $\sum_v\partial f_v= \partial (\sum_v f_v)$, condition $\sum_{v\in V}u(v)=0$ implies that $0\in \partial (\sum_v f_v)(x^\star)$.
Thus,  $x^\star$ is a minimizer of $\sum_v f_v$.
\end{proof}

\subsection{Proof of Proposition~\ref{prop:ac}}
\begin{proof}
Uniqueness of the minimizers is a consequence of the strict convexity, while the equivalence of the three statements
 follows directly from Theorem~\ref{the:min}.
 \end{proof}

\subsection{Proof of Proposition~\ref{prop:MC}}
\begin{proof}
Consider any $x^\star\in \median(x_0)$.
  Consider a bijection $\sigma:\{1\cdots |V|\}\to V$ such that
$(x_0\circ \sigma) (1)\leq \cdots \leq (x_0\circ \sigma) (|V|)$.
Define $u\in \bR^V$ as follows. When $|V|$ is odd, $u(v)$ is equal to
$1$ if $\sigma^{1}(v)<\frac{|V|+1}2$, to $0$ if $\sigma^{-1}(v)=\frac{|V|+1}2$ and to $-1$ otherwise.
When $|V|$ is even,  $u(v)$ is equal to $1$ if $\sigma^{-1}(v)\leq \frac{|V|}2$ and to $-1$ otherwise.
It is straightforward to verify that $u\in\partial F(x^\star 1_V)$.
As $u\in \cU$, one has $\|u\|_*\leq \lambda$. Thus $\partial F(x^\star 1_V)\cap \lambda B_*$ is nonempty.
By Theorem~\ref{the:min}, $x^\star$ is a minimizer of~(\ref{eq:reg}).

By Theorem~\ref{the:min} again, all minimizers of~(\ref{eq:reg}) which belong to $\cC$ necessarily
correspond to minimizers of~(\ref{eq:pb}). Thus the set of minimizers of (\ref{eq:reg}) which belong to $\cC$ is equal to $\median(x_0)1_V$.
It remains to show that~(\ref{eq:reg}) has no minimizers outside the consensus space $\cC$. Denote by $B_0 = \{u\in\mathbb{R}_0^V:\|u\|_\infty\leq 1\}$. From Lemma~\ref{lem:extrem}, its extremal points are given by set $\cU$. Recall that, by definition,
\[
\lambda_0 = \max\{\|u\|_*:u\in\cU\}\;.
\]
Since $B_0$ is a polytope (bounded intersection of halfspaces), it is well known that it is the convex hull of its extremal points. Triangular inequality implies in turn:
\[
\lambda_0 = \max\{\|u\|_*:u\in B_0\}\;.
\]
Now, assume that function $\bs x:V\to\mathbb{R}$ is a minimizer such that $\|\bs x\|_\tv > 0$.  Then for some $g\in\partial F(\bs x)$, one has $g = -\lambda u$ with $u\in\mathbb{R}_0^V$ such that $\langle u,\bs x\rangle = \|x\|_\tv$ and $\|u\|_* \leq 1$ (by Proposition \ref{prop:subgradient-norm}). On the one hand, this implies that $g\in B_0$ since $\|\partial F(\cdot)\|_\infty\leq 1$ and on the other hand it implies that $\|g\|_* = \lambda$. This contradicts $\lambda > \lambda_0$.
\end{proof}

Let us recall a standard definition and derive an easy lemma.
\begin{definition}
  Assume $C$ is a convex set. A point $p\in C$ is said \emph{extremal} when:
  \[
  \lambda\in(0,1)\,,x\in C\,,y\in C\,, p = \lambda x + (1-\lambda)y \Rightarrow x=y
  \]
\end{definition}
\begin{lemma}
 \label{lem:extrem}
  The polyhedral set $B_0 = \{u\in\mathbb{R}_0^V:\|u\|_\infty\leq 1\}$ has extremal point set $\cU$, as defined in Section~\ref{subsec:medcons}.
\end{lemma}
\begin{proof}
  Assume $u$ is an extremal point of $B_0$. And assume by contradiction that there exists $(v,v')\in V^2$ such that: $v\not=v'$ and $\max(|u(v)|,|u(v')|)<1$. Denote by $\epsilon = 1  - \max(|u(v)|,|u(v')|)>0$. Then one has $u= \frac12(u + \epsilon \delta_v - \epsilon\delta_{v'}) + \frac12(u-\epsilon\delta_v+\epsilon\delta_{v'})$, where $\delta_v$ denotes the function from $V$ to $\mathbb{R}$ that takes value $0$ for all $w\not =v$ and value $1$ for $v$, which contradicts extremality. Hence the set $\{v\in V:|u(v)|<1\}$ has at most one element. Considering that $\sum_v u(v)=0$ gives the result.
\end{proof}

\section{Proofs of Section~\ref{sec:algo}}

\subsection{Convergence of Subgradient Algorithms}
\label{app:subg}

Although the result given below is part of the folklore in non-smooth
optimization, the convergence proof is often provided with some
boundedness assumption of subgradients that are in fact not
needed. For the sake of completeness, we provide a self-contained proof.

\begin{assumption}
  \label{assum:f}
  Let $\mathbb{H}$ denote a Euclidean space and $f$ a function from $\mathbb{H}$ to $\mathbb{R}$ such that:
  \begin{enumerate}
    \item $f$ is convex continuous with subgradient $\partial f$.
    \item Lower level sets $L_y = \{x\in\mathbb{H}:f(x)\leq y\}$ are bounded.
    \item There exists $C>0$ such that, $\forall x\in\mathbb{H}$, $\forall g\in\partial f(x)$, $\|g\| \leq C(1+\|x\|)$.
  \end{enumerate}
\end{assumption}

\begin{assumption}
  \label{assum:step}
  Let $\gamma_n$ denote a sequence of positive scalars such that:
  \begin{enumerate}
    \item $\sum_n\gamma_n=+\infty$
    \item $\sum_n\gamma_n^2<+\infty$
  \end{enumerate}
\end{assumption}

\begin{proposition}
\label{prop:subg}
  Under Assumptions~\ref{assum:f} and \ref{assum:step}, any sequence $(x_n)_{n\in\mathbb{N}}$ obeying the subgradient descent scheme:
  \[
  x_{n+1} = x_n - \gamma_{n}g_{n}\;,
  \]
  where $g_n\in\partial f(x_n)$; converges to the set $S = \{x\in\mathbb{H}:f(x) = \inf_{\mathbb{H}} f\}$.
\end{proposition}
\begin{proof}
  Since lower level sets are compact (closed by continuity of $f$ and bounded by assumption) there exists a point $x_*\in\mathbb{H}$ such that $f(x_*) = \min_{x\in\mathbb{H}} f(x)$.
  Denote by $u_n = \|x_n - x_*\|^2$. Thus,
  \[
  u_{n+1} = u_n - 2\gamma_{n}\langle x_n - x_*,g_{n}\rangle + \gamma_{n}^2\|g_{n}\|^2\;.
  \]
  As a consequence of the fundamental property of subgradients, one has $f(x_n) - f(x_*)\leq \langle x_n - x_*,g_n\rangle$. By assumption, $\|g_n\|\leq C(1+\|x_n\|)$. Hence, there exists some constant $M>0$ such that:
  \[
  u_{n+1}\leq u_n - 2\gamma_n (f(x_n)-f(x_*)) + \gamma_{n}^2(M+u_n)
  \]
  Applying Lemma~\ref{lem:deterRS} with $\alpha_n=\gamma_n^2$, $v_n = 2\gamma_n(f(x_n) - f(x_*))$ and $\beta_n = \gamma_n^2M$ yields:
  \begin{enumerate}
  \item $u_n$ converges.
  \item $\sum_n v_n < \infty$.
  \end{enumerate}
  Since $\sum_n \gamma_n =+\infty$ and $f(x_n) - f(x_*)\geq 0$ one necessarily has $\liminf_n f(x_n) - f(x_*) = 0$. Since $u_n$ is bounded, $x_n$ evolves in a compact space and has a convergent subsequence to some point $\tilde x\in\mathbb{H}$. By continuity of $f$, this point $\tilde x$ necessarily belongs to the set $S$. The previous computations are valid with $x_*$ replaced by $\tilde x$. Since Lemma~\ref{lem:deterRS} ensures that $u_n$ should converge, and it has a subsequence converging to $0$, it converges to $0$.
\end{proof}

\begin{lemma}[Deterministic Robbins-Siegmund]
  \label{lem:deterRS}
  Assume $u_n$ and $v_n$ are nonnegative scalar sequences, and $\alpha_n$, $\beta_n$  are sequences such that: $\sum_n |\alpha_n|<\infty$, $\sum_n|\beta_n|<\infty$ and, for all $n$,
  \[
   u_{n+1}\leq (1+\alpha_n)u_n + \beta_n - v_n\;.
  \]
  Then
  \begin{enumerate}
    \item $u_n$ converges to some limit $l\in\mathbb{R}$.
    \item $\sum_n v_n<\infty$.
  \end{enumerate}
\end{lemma}
\begin{proof}
  It is well known that $1+x\leq \exp x$ for all $x\in\mathbb{R}$. Since $u_n$ and $v_n$ are non-negative, one has: $u_{n+1}\leq \exp (\alpha_n) u_n +\beta_n$. Which iteratively implies: $u_n \leq \exp(\sum_{i=0}^{n-1}\alpha_i)u_0 + \sum_{i=0}^{n-1}\exp(\sum_{j=i+1}^{n-1}\alpha_j)\beta_i$. Since $\sum_n |\alpha_n|<\infty$, for all integers $k$ and $l$ such that $k<l$, $\sum_{n=k}^l\alpha_n\leq \sum_{n=0}^\infty|\alpha_n|$ and $\exp(\sum_{n=k}^l\alpha_n)\leq \exp(\sum_{n=0}^\infty|\alpha_n|)<\infty$. Thus, $\sum_n|\beta_n|<\infty$, and it holds that $|u_n|\leq M= \exp(\sum_{n=0}^\infty|\alpha_n|)(|u_0|+\sum_{n=0}^\infty|\beta_n|)<\infty$. Hence,
  $|u_{n+1}-u_n|\leq M|\alpha_n|+|\beta_n|$ and $\sum_n|u_{n+1}-u_n|<\infty$. Thence, as stated in the lemma:
  \begin{enumerate}
  \item Sequence $u_n = u_0+\sum_{m=0}^{n-1}(u_{m+1}-u_m)$ is convergent.
  \item $\sum v_n\leq \sum_n|u_{n+1}-u_n|+ M\sum_n|\alpha_n| + \sum|\beta_n|<\infty$.
  \end{enumerate}

\end{proof}

\subsection{\modif{Derivation of the ADMM}}
\label{app:admm}

using the notation introduced in Section \ref{sec:admm},
the update equation of $x_n$ simplifies to:
\begin{eqnarray*}
x_{n+1}(v) &=& \arg\min_{x} f_v(x) + \sum_{w\sim v} -\eta_n(w,v)x +  \frac\rho2(z_n(w,v)-x)^2 \\
&=& \prox_{f_v,\,\rho d(v)}\left(\tilde z_n(v)+\frac{\tilde \eta_n(v)}\rho\right)
\end{eqnarray*}
where $d(v)$ is the degree of $v$, $\tilde z_n(v) = d(v)^{-1}\sum_{w\sim v} z_n(w,v)$
and $\tilde \eta_n(v) = d(v)^{-1}\sum_{w\sim v} \eta_n(w,v)$. The update equation of sequence $z_n$ reduces to:
\begin{eqnarray*}
  \left(\begin{array}[h]{@{}c@{}}
z_{n+1}(v,w) \\
z_{n+1}(w,v)
\end{array}\right)
&=& \arg\min_{(z_1,z_2)}  S_\rho(z_1,z_2;\eta_n(v,w),\eta_n(w,v),x_{n+1}(w),x_{n+1}(v))
\end{eqnarray*}
where $S_\rho(z_1,z_2;\eta_1,\eta_2,x_1,x_2) =\lambda\,|z_1-z_2| + \eta_1z_1+ \eta_2z_2+  \frac\rho2(z_1-x_1)^2+  \frac\rho2(z_2-x_2)^2$.
Minimization of $S_\rho$ w.r.t. $(z_1,z_2)$ yields
\begin{eqnarray}
\frac{z_1+z_2}2 &=& \frac{x_1+x_2}2-\frac{\eta_1+\eta_2}{2\rho} \label{eq:sumZ}\\
\frac{z_1-z_2}2 &=& \soft_{\lambda/\rho}\left(\frac{x_1-x_2}2+\frac{\eta_2-\eta_1}{2\rho}\right) \label{eq:difZ}
\end{eqnarray}
where $\soft_\omega(x) = \sign(x)\cdot\max(|x|-\omega,0)$ is the soft-thresholding function.
Using the update equation of $\eta_n$, we obtain
$$
\eta_{n+1}(v,w) + \eta_{n+1}(w,v) = \eta_{n}(v,w)+ \eta_{n}(w,v)+
\rho\left(z_{n+1}(v,w)+z_{n+1}(w,v)-x_{n+1}(v)-x_{n+1}(w)\right)
$$
which by~(\ref{eq:sumZ}) implies that $\eta_{n+1}(v,w) + \eta_{n+1}(w,v)=0$.
Therefore, equation (\ref{eq:sumZ}) implies that for each $n$,
$$
{z_{n}(v,w)+z_{n}(w,v)} ={x_{n}(v)+x_{n}(w)}
$$
We set $\delta_n(v,w) = \frac 12(z_{n}(v,w)-z_{n}(w,v))$.
Equation~(\ref{eq:difZ}) implies that
$$
\delta_{n+1}(v,w) =  \soft_{\lambda/\rho}\left(\frac{x_{n+1}(w)-x_{n+1}(v)}2-\frac{\eta_{n}(v,w)}{\rho}\right)
$$
Using again the update equation of $\eta_n$, we obtain
$$
\eta_{n+1}(v,w)  = \eta_{n}(v,w)+
\rho\left(\delta_{n+1}(v,w)-\frac 12(x_{n+1}(w)-x_{n+1}(v))\right)
$$
Set $\beta_{n}(v,w)=2\eta_{n}(v,w)/\rho$. After some algebra,
$$
\beta_{n+1}(v,w)  = \proj_{[-2\lambda/\rho,2\lambda/\rho]}\left(\beta_{n}(v,w)+x_{n+1}(v)-x_{n+1}(w)\right)\ .
$$
Finally, we simplify the update equation in $x_n$ as follows. Using
$z_n(w,v) = \frac 12(x_{n}(w)+x_{n}(v)) + \delta_n(w,v)$,
we obtain
$
\tilde z_n(v) =\frac 12(x_{n}(v)+\tilde x_{n}(v)) + \tilde\delta_n(v)
$
where $\tilde x_{n}(v) = d(v)^{-1}\sum_{w\sim v}x_n(w)$
and $\tilde\delta_n(v)  = d(v)^{-1}\sum_{w\sim v}\delta_n(w,v)$.
By summing equality $2\delta_{n}(w,v)=\beta_{n}(w,v)  - \beta_{n-1}(w,v) +x_{n}(v)-x_{n}(w)$
w.r.t $w\sim v$, one has $2\tilde\delta_n(v) = \tilde\beta_{n}(v)  - \tilde\beta_{n-1}(v) +x_{n}(v)-\tilde x_{n}(v)$
where $\tilde \beta_{n}(v) = d(v)^{-1}\sum_{w\sim v}\beta_{n}(w,v)$.
Thus,
$
\tilde z_n(v) =x_{n}(v)+\frac 12( \tilde\beta_{n}(v)  - \tilde\beta_{n-1}(v) )\,.
$
Finally,
$$
x_{n+1}(v) = \prox_{f_v,\,\rho d(v)}\left(x_{n}(v)+\textstyle\frac 32 \tilde\beta_{n}(v)  - \frac 12\tilde\beta_{n-1}(v) \right)\ .
$$
Setting $\mu_{n+1}(v,w) = \beta_{n}(v,w)$ yields exactly { Algorithm 2}.

We now assume that the graph $G$ is $d$-regular and show that the average over the network is preserved. We  consider here on the average consensus case.
  Computing the derivative $\frac{\partial \mathcal{L}}{\partial x}$ which exists since $\mathcal{L}$ is smooth in $x$; one gets the necessary condition:
  \[
  \forall v\in V, \mskip 20mu x_{n+1}(v) - x_0(v) - \sum_{v\sim w}\eta_n(w,v) + \rho\sum_{v\sim w}\left(x_{n+1}(v) -  z_n(w,v)\right) = 0
  \]
  Then, summing up, and using the fact that (i), when an edge $(v,w)$ belongs to $\overset{\rightleftharpoons}E$, the edge $(w,v)$ also belongs to $\overset{\rightleftharpoons}E$, and (ii) $\eta_{n+1}(v,w) = -\eta_{n+1}(w,v)$, along with (iii) $z_n(v,w) + z_n(w,v) = x_n(v)+x_n(w)$:
  \[
  \sum_v \left(1+2\rho d(v)\right)x_{n+1}(v)  = \sum_v x_0(v) + 2\rho\sum_v d(v) x_n(v)
  \]
  For a $d$-regular graph, we have, by induction: $\forall n\geq 0$, $\bar x_n=\bar x_0$.

\section{Proofs of Section~\ref{sec:stubborn}}


\subsection{Proof of Proposition~\ref{prop:deterg-cvg}}
\label{app:deter-cvg}

Equation $ x_{n+1} =  W x_n$ writes $ x_{n+1}^R =  W^R  x_n^R
+  W^S x^S$. Repeating the argument, we get:
\[
 x_n^R = ( W^R)^n x_0^R + \sum_{k=0}^{n-1}( W^R)^k W^S x^S
\]

Lemma~\ref{lem:invertibility} below shows that $\rho(W^R)<1$.  From $\rho(W^R) < 1$ we
deduce that $( W^R)^n x_0^R$ tends to $0$ and $\sum_{k=0}^{n-1}(
W^R)^k W^S x^S$ tends to $( I- W^R)^{-1} W^S x^S$.

\begin{lemma}
  \label{lem:invertibility}
  Under Assumption~\ref{assum:stubgoss_rand}, for each complex number $z$ with
  $|z|\geq 1$, matrix $z I -  W^R$ is invertible.
\end{lemma}

\begin{proof}
Assume that $z I -  W^R$ be not invertible, then there
exists a vector $ x\not=0$ such that $ W^R x = z x$. Let us denote
by $v_0\in R$ a node such that $| x(v)|$ be maximum. We have,
$z\bs x(v_0) = \sum_{v\in R} w^R(v_0,v) x(v)$.
Hence, $|z|| x(v_0)| \leq \sum_{v\in R} w^R(v_0,v)| x(v)|$.  By definition
of $v_0$: $|z|| x(v_0)| \leq (\sum_{v\in R} w^R(v_0,v))| x(v)|$. Since
$ W  1 =  1$, we have $\sum_{v\in R}w^R(v_0,v)\leq 1$. Thence,
necessarily,
$\sum_{v\in R}w^R(v_0,v) = 1$ which proves that all neighbors of
$v_0$ for $E'$ are in $R$, which in turn implies that $|z|| x(v_0)| =
\sum_{v\in R} w^R(v_0,v)| x(v)|$. Moreover, $w^R(v_0,v) > 0$ implies that
$ x(v) = z x(v_0)$, otherwise the equality $z x(v_0) = \sum_{v\in
  R}w^R(v_0,v) x(v)$ would be violated. So one can repeat the argument with
all the neighbors of $v_0$: all their own neigbors are necessarily in $R$ and
eventually all the connected component containing $v_0$ lies in $R$, which
completes the proof.
\end{proof}

\subsection{Proof of Theorem~\ref{the:robTVGA}}
\label{app:robTVGA}

We are going to use Theorem  \ref{the:min}   where $f_v(x)$ is given by $\frac{1}{2} (x - x_0(v))^2 + \lambda \sum_{\substack{w \in S}} |{ x -  x_0(w)}|$.

Let us first consider the case  $\overline { x}_0^{R} + \lambda |S|  < a$.
The subdifferential $\partial F(x^\star 1_V)$ for  $x^* =  \overline { x}_0^{R} + \lambda |S|$ is clearly given by the singleton $\{x^* 1_R  - x_0^R  - \lambda |S|
 1_R \}  = \{  \overline { x}_0^{R} 1_R - x_0^R  \}$.  Since $\lambda \geq     ||x_0^{R} - \overline {x}_0^{R} 1_R ||_*$, the second condition of   Theorem  \ref{the:min} is satisfied implying that $x^* 1_V$ is a minimizer of \eqref{eq:13}. The strict convexity of \eqref{eq:13} implies the uniqueness of the minimizer.

The case $\overline { x}_0^{R} - \lambda |S|  > a$ is exactly handled  in the same way.

Let us now assume that $|\overline { x}_0^{R} - a| \leq  \lambda |S|$ and take
$x^* = a$. Using the fact that the subdifferential of $|x|$ for $x = 0$ is given by the interval $[-1,1]$, we deduce that  $$\partial f_v(x^\star = a) =  \lambda \times  \{ \sum_{\substack{w \in S \\ |\zeta_{v,w}| \leq 1} } \zeta_{v,w}     \}  + a - x_0(v).$$
Let  $\zeta  =  \frac{\overline { x}_0^{R} - a}{\lambda |S|}$ and let us take $\zeta_{v,w} = \zeta$ for $v \in R$ and $w \in S$.
$\zeta_{v,w}$ clearly belongs to the interval $[-1,1]$. Moreover, using the fact that $
\partial F(x^\star 1_R) = \left\{ u\in \bR^R : \forall v\in R,\, u(v)\in \partial f_v(x^\star)\right\}$, we deduce that  $\partial F(x^\star 1_V)$ contains the vector
$x^* 1_R - x_0^R + \lambda \zeta |S| 1_R$.  Using the definition of $\zeta$, we get that $\overline { x}_0^{R} 1_R - x_0^R $ belongs to $\partial F(x^\star 1_V)$. Using again the assumption $\lambda \geq     ||x_0^{R} - \overline {x}_0^{R} 1_R ||_*$  and Theorem  \ref{the:min} we get that $x^* 1_V$ is a minimizer of \eqref{eq:13}. Uniqueness of the minimizer is still implied by the strict convexity of \eqref{eq:13}.

\end{document}